\documentclass[preprint]{imsart}

\newenvironment{hiddenproof}{Proof}{\qed\\}

\newif\ifhideproofs%
\hideproofstrue

\let\oriwrite=\write
\ifhideproofs%
  \usepackage{environ}%
  \NewEnviron{hide}{
   \setbox0=\vbox{\def\write{\immediate\oriwrite}\BODY}
  }
\fi

\usepackage[utf8]{inputenc} 
\usepackage{amsmath,amsthm,verbatim,amssymb,amsfonts,amscd,graphicx}
\usepackage[colorlinks,citecolor=blue,urlcolor=blue]{hyperref}

\usepackage{mathtools}

\usepackage{geometry} 
\usepackage{mathabx} 

\usepackage{graphicx} 

\usepackage{booktabs} 
\usepackage{array} 
\usepackage{paralist} 
\usepackage{verbatim} 
\usepackage{subfig} 
\usepackage[utf8]{inputenc}
\usepackage[english]{babel}

\usepackage{sectsty}
\allsectionsfont{\sffamily\mdseries\upshape} 

\renewcommand{\exp}[1]{\operatorname{exp}\left(#1\right)} 
\providecommand{\argmax}{\mathop\mathrm{arg max}} 
\providecommand{\argmin}{\mathop\mathrm{arg min}}

\providecommand{\tr}{\mathop\mathrm{tr}}

\newcommand\var{\mathrm{Var}}
\newcommand\cov{\mathrm{cov}}
\newcommand{\eps}{\varepsilon}
\newcommand{\diam}{\mathrm{diam}}
\newcommand{\wh}{\widehat}
\newcommand{\wt}{\widetilde}

\usepackage[nottoc,notlof,notlot]{tocbibind} 
\usepackage{natbib}
\usepackage{bbm}
\usepackage[titles,subfigure]{tocloft} 
\usepackage{comment}

\newcommand{\p}{{\rm I}\kern-0.18em{\rm P}}
\newcommand{\1}{{\rm 1}\kern-0.24em{\rm I}}
\newcommand{\E}{{\rm I}\kern-0.18em{\rm E}}
\newcommand{\R}{{\rm I}\kern-0.18em{\rm R}}

\newtheorem{theorem}{Theorem}
\newtheorem{lemma}{Lemma}

\newtheorem{proposition}{Proposition}

\newtheorem{remark}{Remark}
\newtheorem{fact}{Fact}

\newcommand{\rom}[1]{\uppercase\expandafter{\romannumeral #1\relax}}
\newcommand{\be}{\begin{equation*}}
\newcommand{\ee}{\end{equation*}}
\newcommand{\ben}{\begin{equation}}
\newcommand{\een}{\end{equation}}
\newcommand{\bmln}{\begin{multline*}}
\newcommand{\emln}{\end{multline*}}

\def\l{\left}
\def\r{\right}
\newcommand{\m}{\mathcal}
\newcommand{\mb}{\mathbb}

\newcommand{\bL}{\widebar L}
\newcommand{\hL}{\widehat{L}}
\newcommand{\htheta}{\widehat \theta}

\newcommand{\pr}[1]{\mathbb{P}{\left(#1\right)}}

\newcommand{\pd}{\partial}

\def\ml#1{\begin{multline*}{#1}\end{multline*}}
\def\mln#1{\begin{multline}{#1}\end{multline}}

\newtheorem{innercustomthm}{Assumption}
\newenvironment{customassumption}[1]
  {\renewcommand\theinnercustomthm{#1}\innercustomthm}
  {\endinnercustomthm}

\allowdisplaybreaks

\begin{document}

\begin{frontmatter}
\title{Asymptotic normality of robust risk minimizers}
\runauthor{ }
\runtitle{Robust ERM}

\begin{aug}
\author{\fnms{Stanislav} \snm{Minsker}\ead[label=e1,mark]{minsker@usc.edu} \thanksref{e1,a}}

\address[a]{Department of Mathematics, University of Southern California \\
\printead{e1}}

\thankstext{t3}{Author acknowledges support by the National Science Foundation grants DMS CAREER-2045068 and CCF-1908905.}
\end{aug}

\begin{abstract}
This paper investigates asymptotic properties of algorithms that can be viewed as robust analogues of the classical empirical risk minimization. 
These strategies are based on replacing the usual empirical average by a robust proxy of the mean, such as the (version of) the median of means estimator. It is well known by now that the excess risk of resulting estimators often converges to zero at optimal rates under much weaker assumptions than those required by their ``classical'' counterparts. 
However, less is known about the asymptotic properties of the estimators themselves, for instance, whether robust analogues of the maximum likelihood estimators are asymptotically efficient. We make a step towards answering these questions and show that for a wide class of parametric problems, minimizers of the appropriately defined robust proxy of 
the risk converge to the minimizers of the true risk at the same rate, and often have the same asymptotic variance, as the estimators obtained by minimizing the usual empirical risk. 
\end{abstract}


\end{frontmatter}

\mathtoolsset{showonlyrefs=true}

\section{Introduction.}
\label{sec:intro}

The concept of robustness addresses stability of statistical estimators under various forms of perturbations, such as the presence of corrupted/atypical observations (``outliers'') in the data. 
The questions related to robustness in the framework of statistical learning theory have recently seen a surge in interest, both from the theoretical and practical perspectives, and resulted in the development of novel algorithms. 
These new robust algorithms are characterized by the fact that they provably work under minimal assumptions on the underlying data-generating mechanism, often requiring the existence of moments of low order only. 
Majority of the existing works focused on the upper bounds for the risk of the estimators (such as the classification or prediction error) produced by the algorithms, while in this paper we are interested in the asymptotic properties of the estimators themselves. The asymptotic viewpoint allows one to gauge efficiency of the estimators and understand the magnitude of constants appearing in the bounds, as opposed to just studying the form of dependence of the bounds on the parameters of interest (sample size, dimension, etc.)

Next, we introduce the mathematical framework used in the exposition. 
Let $(S,\m S)$ be a measurable space, and let $X\in S$ be a random variable with distribution $P$. 
Suppose that $X_1,\ldots,X_N$ are i.i.d. copies of $X$. Moreover, assume that $\m L = \l\{ \ell(\theta,\cdot), \ \theta\in \Theta\subseteq \mb R^d \r\}$ is a class of measurable functions from $S$ to $\mb R$ indexed by an open subset of $\mb R^d$. 
Population versions of many estimation problems in statistics and statistical learning, such as maximum likelihood estimation and regression, can be formulated as risk minimization of the form 
\ben
\label{eq:problem}
\mb E \, \ell(\theta,X) \to \min_{\theta\in \Theta}.
\een
In particular, when $\l\{ p_\theta, \ \theta\in \Theta\r\}$ is a family of probability density functions with respect to some $\sigma$-finite measure $\mu$ and $\ell(\theta,\cdot) = -\log p_\theta(\cdot)$, the resulting problem corresponds to maximum likelihood estimation. 
In what follows, we will set $L(\theta)$ to be the risk associated with the parameter $\theta$, namely $L(\theta) = \mb E \ell(\theta,X)$. 
Throughout the paper, we will assume that the minimum in problem \eqref{eq:problem} is attained at a unique point $\theta_0 \in \Theta$. 
The true distribution $P$ is typically unknown, and an estimator of $\theta_0$ is obtained via minimizing the \emph{empirical risk}, namely,
\ben
\label{eq:erm-standard}
\wt \theta_N := \argmin_{\theta\in \Theta} L_N(\theta), 
\een
where $L_N(\theta) := \frac{1}{N}\sum_{j=1}^N \ell\l(\theta,X_j\r)$. 
If the marginal distributions of the process $\{\ell(\theta,\cdot), \ \theta\in \Theta\}$ are heavy-tailed, meaning that they possess finite moments of low order only, then the error $|L_N(\theta)-L(\theta)|$ can be large with non-negligible probability, motivating the need for alternative proxies for the risk $L(\theta)$. Another scenario of interest corresponds to the \emph{adversarial contamination} framework, where the initial dataset of cardinality $N'$ is merged with a set of $\m O<N'$ outliers generated by an adversary who has complete knowledge of the underlying distribution and an opportunity to inspect the data, and the combined dataset of cardinality $N=N'+\m O$ is presented to the algorithm responsible for constructing the estimator of $\theta_0$. In what follows, the proportion of outliers will be denoted by 
$\kappa:=\frac{\m O}{N}$. 
Similarly to the heavy-tailed scenario, the empirical loss $L_N(\theta)$ is not a robust proxy for $\mb E \ell(\theta,X)$ in this case, therefore estimation and inference results based on minimizing $L_N(\theta)$ may be unreliable. 

One may approach the problem of estimating $\theta_0$ robustly from different angles. 
One class of popular methods consists of robust versions of the gradient descent algorithm for the optimization problem \eqref{eq:problem}, where the gradient $\nabla L(\theta_k)$ is estimated on each iteration $k$; for example, this approach has been explored by \cite{prasad2018robust,chen2017distributed,yin2018byzantine,alistarh2018byzantine}, among others. 
Another technique (the one that we investigate in this paper) is based on replacing the average $L_N(\cdot)$ by a robust proxy of $L(\theta)$.  
Its advantage is the fact that we only need to estimate a real-valued quantity $L(\theta)$, as opposed to the high-dimensional gradient vector $\nabla L(\theta)$. 
On the other hand, favorable properties, such as convexity, that are ``inherited'' by the formulation \eqref{eq:erm-standard} from \eqref{eq:problem}, are usually lost in this case. 
Several representative papers that explore this direction include the works by \cite{audibert2011robust,lerasle2011robust,brownlees2015empirical,lugosi2016risk,lecue2017robust,cherapanamjeri2019algorithms,mathieu2021excess}; also, see an excellent survey paper by \cite{lugosi2019mean}.
Instead of the empirical risk $L_N(\theta)$, these works employ robust estimators of the risk such as the median of means estimator \citep{Nemirovski1983Problem-complex00,alon1996space,devroye2016sub} or Catoni's estimator \citep{catoni2012challenging}. 

In this paper, we study estimators based on the modification of the median of means principle introduced in \citep{minsker2017distributed} combined with the idea behind the ``median of means tournaments'' \citep{lugosi2016risk} and the closely related ``min-max'' robust estimators \citep{audibert2011robust,lecue2017robust}. The latter are based on an observation that $\theta_0$ can be alternatively obtained via 
\ben
\label{eq:difference}
\theta_0 = \argmin_{\theta\in\Theta}\max_{\theta'\in \Theta}\l( L(\theta) - L(\theta')\r).
\een
Therefore, an estimator of $\theta_0$ can be constructed by replacing the difference $L(\theta,\theta'):=L(\theta) - L(\theta')$ by its robust proxy constructed as follows. Let $k\leq N/2$ be an integer, and assume that $G_1,\ldots,G_k$ are disjoint subsets of the index set $\{1,\ldots,N\}$ of cardinality $|G_j| = n\geq \lfloor N/k\rfloor$ each. For $\theta\in \Theta$, let 
\[
\bL_j (\theta) := \frac{1}{n}\sum_{i\in G_j} \ell(\theta, X_i) 
\]
be the empirical risk evaluated over the subsample indexed by $G_j$. Assume that $\rho:\mb R\mapsto \mb R_+$ is a convex, even function that is increasing on $(0,\infty)$ and such that its (right) derivative is bounded. Let $\{\Delta_n\}_{n\geq 1}$ be a non-decreasing positive sequence of ``scaling factors'' such that $\Delta_n = o(\sqrt{n})$ and $\Delta_\infty := \lim_{n\to\infty}\Delta_n\in (0,\infty]$, and define  
\ben
\label{eq:M-est-2}
\hL(\theta,\theta')\in\argmin_{z\in \mb R}\sum_{j=1}^k \rho\l(\sqrt{n}\,\frac{\bL_j (\theta) - \bL_j(\theta') - z}{\Delta_n}\r).
\een
For example, the choice $\Delta_n \asymp \log(n)$ suffices for all results of the paper to hold (in fact, it suffices for $\Delta_\infty$ to be a sufficiently large constant); we will make a remark regarding the practical aspects of setting $\Delta_n$ below. The estimator $\hL(\theta,\theta')$ is what we referred to as the robust proxy of $L(\theta,\theta')$, where robustness is justified by the fact that the error $\l|\hL(\theta,\theta') - L(\theta,\theta')\r|$ satisfies non-asymptotic exponential deviation bounds under minimal assumptions on the tails of the random variables 
$\ell(\theta,X)-\ell(\theta',X)$ and the ability of $\hL(\theta,\theta')$ to resist adversarial outliers. 
For example, Theorem 3 in \citep{minsker2017distributed} essentially states that whenever $\Delta_n\gtrsim  \var^{1/2}\l( \ell(\theta,X)-\ell(\theta',X)\r)$ and for all $s\lesssim k$, 
\[
\l| \hL(\theta,\theta') - L(\theta,\theta') \r| \lesssim \sigma(\theta,\theta')\sqrt{\frac{s}{N}} + \Delta_n\l(\frac{k}{N} + \frac{\m O\sqrt{n}}{N}\r)
\] 
with probability at least $1-e^{-s}$, assuming that $\mb E |\ell(\theta,X)-\ell(\theta',X)|^3<\infty$ and where $\lesssim$ denotes the inequality up to absolute numerical constants; similar guarantees also hold uniformly over $\theta,\theta'\in \Theta$; note that setting $\Delta_n \asymp  \sigma(\theta,\theta')$ yields the most robust estimator. Given the robust proxy $\wh L(\theta,\theta')$ of $L(\theta,\theta')$, an analogue of the classical empirical risk minimizer $\wt\theta_N$ can be obtained via
\ben
\label{eq:M-est-erm2}
\htheta_{n,k} = \argmin_{\theta\in\Theta}\sup_{\theta'\in \Theta} \hL(\theta,\theta').
\een
Simple sufficient conditions for the existence of $\htheta_{n,k}$ are discussed in the appendix; in principle, one could consider  near-minimizers instead, however, we avoid this route due to the extra layer of technicalities it brings. The idea behind considering differences of the risks and defining $\theta_0$ via \eqref{eq:difference} is related to the fact that the estimators \eqref{eq:M-est-2} of $L(\theta)$, unlike their traditional counterparts $L_N(\theta)$, are non-linear: if we set $\hL(\theta) = \argmin_{z\in \mb R}\sum_{j=1}^k \rho\l(\sqrt{n}\,\frac{\bL_j (\theta) - z}{\Delta_n}\r)$, then $\hL(\theta,\theta')\ne \hL(\theta) - \hL(\theta')$. 

Related approaches based on direct minimization of $\hL(\theta)$ have been previously investigated by \cite{brownlees2015empirical,holland2017robust,lecue2018robust,mathieu2021excess}, where the main object of interest was the excess risk $\m E(\wh\theta_{n,k}):=L(\wh\theta_{n,k}) - L(\theta_0)$. It has been  recognized however that non-linearity of $\hL(\theta)$ often results in sub-optimal rates, while the tournament-type procedures avoid these shortcomings. 
In the present work, we will be interested in the asymptotic behavior of the error $\wh\theta_{n,k} - \theta_0$, rather than the excess risk: in particular, we will establish asymptotic normality of the sequence $\sqrt{N}\l( \wh\theta_{n,k} - \theta_0 \r)$ and demonstrate that robust estimators can still be efficient under essentially the same set of sufficient conditions as required by the standard M-estimators \citep{van2000asymptotic}.   
Let us mention that the nonlinear nature of the estimator $\hL(\theta,\theta')$ makes the proofs more technical compared to the classical theory of M-estimators based on empirical risk minimization. Our arguments rely on Bahadur-type representations for $\wh L(\theta,\theta')$ whose remainder terms admit tight uniform bounds.



\subsection{Notation.}
\label{section:def}

Absolute constants will be denoted $c,c_1,C,C_1,C'$, etc., and may take different values in different parts of the paper. 
Given $a,b\in \mb R$, we will write $a\wedge b$ for $\min(a,b)$ and $a\vee b$ for $\max(a,b)$. 
For a function $f:\mb R^d\mapsto\mb R$, define 
\[
\argmin_{y\in\mb R^d} f(y) := \{y\in\mb R^d: f(y)\leq f(x)\text{ for all }x\in \mb R^d\},
\]
and $\|f\|_\infty:=\mathrm{ess \,sup}\{ |f(y)|: \, y\in \mb R^d\}$. Moreover, $\mathrm{Lip}(f)$ will stand for the Lipschitz constant of $f$; if $d=1$ and $f$ is $m$ times differentiable, $f^{(m)}$ will denote the $m$-th derivative of $f$. 
For a function $g(\theta,x)$ mapping $\mb R^d\times \mb R$ to $\mb R$, $\pd_\theta g$ will denote the vector of partial derivatives with respect to the coordinates of $\theta$; similarly, $\pd^2_\theta g$ will denote the matrix of second partial derivatives.

For $x\in \mb R^d$, $\|x\|$ will stand for the Euclidean norm of $x$, $\|x\|_\infty:=\max_j |x_j|$, and for a matrix $A\in \mb R^{d\times d}$, $\|A\|$ will denote the spectral norm of $A$. We will frequently use the standard big-O and small-o notation, as well as their in-probability siblings $o_P$ and $O_P$. For vector-valued sequences $\{x_j\}_{j\geq 1}, \ \{y_j\}_{j\geq 1}\subset \mb R^d$, expressions $x_j = o(y_j)$ and $x_j = O(y_j)$ are assumed to hold coordinate-wise. For a square matrix $A\in \mb R^{d\times d}$, $\tr A:=\sum_{j=1}^d A_{j,j}$ denotes the trace of $A$. 

Given a function $g:\mb R\mapsto \mb R$, measure $Q$ and $1\leq p<\infty$, we set $\|g\|^p_{L_p(Q)}:= \int_\mb R |g(x)|^p dQ$. 
For i.i.d. random variables $X_1,\ldots,X_N$ distributed according to $P$, $P_N:=\frac{1}{N}\sum_{j=1}^N \delta_{X_j}$ will stand for the empirical measure; here, $\delta_{X}(g):=g(X)$. The expectation with respect to a probability measure $Q$ will be denoted $\mb E_Q$; if the measure is not specified, it will be assumed that the expectation is taken with respect to $P$, the distribution of $X$. Given $f:S\mapsto \mb R^d$, we will write $Qf$ for $\int f dQ\in \mb R^d$, assuming that the last integral is calculated coordinate-wise. 
For $\theta,\theta'\in \Theta$, let $\sigma^2(\theta,\theta') = \var\l( \ell(\theta,X)-\ell(\theta',X)\r)$ and for $\Theta'\subseteq \Theta$, define 
$\sigma^2(\Theta'):=\sup_{\theta,\theta'\in \Theta'} \sigma^2(\theta,\theta')$. 

Finally, we will adopt the convention that the infimum over the empty set is equal to $+\infty$.  
Additional notation and auxiliary results are introduced on demand.

\section{Statements of the main results.}

We begin by listing the assumptions on the model; these conditions are similar to the standard assumptions made in the parametric estimation  framework \citep{van2000asymptotic,wellner1}. 
The first assumption lists the requirements for the loss function $\rho$ (note that the choice of this function is completely determined by the statistician).
\begin{customassumption}{1}
\label{ass:1}
The function $\rho: \mb R\mapsto \mb R$ is convex, even, and such that
\begin{enumerate}
\item[(i)] $\rho'(z)=z$ for $|z|\leq 1$ and $\rho'(z)=\mathrm{const}$ for $z\geq 2$.
\item[(ii)] $z - \rho'(z)$ is nondecreasing; 
\item[(iii)] $\rho^{(5)}$ is bounded and Lipschitz continuous.
\end{enumerate}
\end{customassumption}
An example of a function $\rho$ satisfying required assumptions is given by ``smoothed'' Huber's loss defined as follows. Let
\[
H(y)=\frac{y^2}{2} I\{|y|\leq 3/2\} + \frac{3}{2}\l(|y| - \frac{3}{4}\r) I\{|y|>3/2\}
\] 
be the usual Huber's loss. Moreover, let $\psi$ be the mollifier $\psi(x) = C \exp{ -\frac{4}{1 - 4x^2}} \, \l\{ |x|\leq \frac{1}{2} \r\}$ where $C$ is chosen so that $\int_\mb R \psi(x)dx = 1$. Then $\rho$ given by the convolution $\rho(x) = (h\ast \psi)(x)$ satisfies assumption \ref{ass:1}. 

\begin{remark}
The classical median of means estimator \citep{Nemirovski1983Problem-complex00,alon1996space} corresponds to the choice $\rho(x)=|x|$ that does not satisfy smoothness assumptions imposed above. 
Asymptotic behavior of the estimators corresponding to this loss is left as an open problem; numerical evidence suggesting that asymptotic normality does not hold in this case is presented in \citep{yao2022median}.
\end{remark}

\begin{customassumption}{2}
\label{ass:2}
The Hessian $\pd^2_\theta L(\theta_0)$ exists and is strictly positive definite. 
\end{customassumption}
This assumption ensures that in a sufficiently small neighborhood of $\theta_0$, 
$c(\theta_0)\|\theta-\theta_0\|^2 \leq L(\theta) - L(\theta_0) \leq C(\theta_0)\|\theta-\theta_0\|^2$ for some $0<c(\theta_0)\leq C(\theta_0)<\infty$. 
The following two conditions allow one to control the ``complexity'' of the class $\{\ell(\theta,\cdot), \ \theta\in\Theta\}$.
\begin{customassumption}{3}
\label{ass:3}
For every $\theta\in \Theta$, the map $\theta'\mapsto\ell(\theta',x)$ is differentiable at $\theta$ for $P$-almost all $x$ (where the exceptional set of measure $0$ can depend on $\theta$), with derivative 
$\pd_\theta \ell(\theta,x)$. Moreover, $\forall \theta\in \Theta$, the envelope function
$\m V(x;\delta):=\sup_{\|\tilde\theta - \theta\|\leq\delta} \l\| \pd_\theta \ell(\tilde\theta,x)\r\|$ of the class $\l\{ \pd_\theta \ell(\tilde\theta,\cdot): \|\tilde\theta - \theta\|\leq \delta\r\}$ satisfies $\mb E \m V^{2}(X;\delta)<\infty$ for sufficiently small $\delta=\delta(\theta)$. 
\end{customassumption}
An immediate implication of this assumption is the fact that the function $\theta\mapsto \ell(\theta,x)$ is locally Lipschitz. It other words, for any $\theta\in \Theta$, there exists a ball $B(\theta,r(\theta))$ of radius $r(\theta)$ such that for all $\theta_1, \, \theta_2\in B(\theta,r(\theta))$, 
$\lvert \ell(\theta_1,x) - \ell(\theta_2,x)\rvert \leq \m V(x;r(\theta))\|\theta_1 - \theta_2\|$. In particular, this condition suffices to prove consistency of the estimators considered in this work and is similar to the classical assumptions used in the analysis of M-estimators, e.g. see the book by \cite{van2000asymptotic}.
The final assumption that we impose allows us to treat non-compact parameter spaces. 
Essentially, we require that the estimator $\wh\theta_{n,k}$ defined via \eqref{eq:M-est-erm2} belongs to a compact set of sufficiently large diameter with high probability, namely,
\begin{equation*}
\lim_{R\to\infty} \limsup_{n,k\to\infty} \pr{\l\| \wh\theta_{n,k} - \theta_0 \r\|\geq R} = 0 \text{ and }
\end{equation*}
The following condition is sufficient for the display above to hold:
\begin{customassumption}{4}
\label{ass:4}
Given $t,R>0$ and a positive integer $n$, define
\[
B(n,R,t) := \pr{\inf_{\theta\in \Theta, \, \|\theta - \theta_0\|\geq R} \frac{1}{n}\sum_{j=1}^n \ell(\theta,X_j) <  \mb E\ell(\theta_0,X) + t }.
\] 
Then $\lim_{R\to\infty} \limsup_{n\to\infty} B(n,R,t) = 0$ for some $t>0$.
\end{customassumption}
Requirements similar to assumption \ref{ass:4} are commonly imposed in the classical framework of M-estimation, \citep[e.g see][]{van2000asymptotic}. 
Of course, when $\Theta$ is compact, assumption \ref{ass:4} holds automatically; another general scenario when assumption \ref{ass:4} is true occurs if the class $\l\{ \ell(\theta,\cdot): \ \theta\in \Theta\r\}$ is Glivenko-Cantelli \citep{wellner1}. Otherwise, it can usually be verified on a case-by-case basis. 
For instance, consider the framework of linear regression, where the data consist of i.i.d. copies of the random couple $(Z,Y)\in \mb R^d\times \mb R$ such that $Y = \langle Z,\theta_\ast\rangle+\eps$ for some $\theta_\ast\in \mb R^d$ and a noise variable $\eps$ that is independent of $Z$ and has variance $\sigma^2$. Moreover, assume that $Z$ is centered and has positive definite covariance matrix $\Sigma$. 
In this case, $\ell(\theta,Z,Y) = \l( Y - \langle Z,\theta \rangle\r)^2 $, and it is easy to see that $\frac{1}{n}\sum_{j=1}^n \ell(\theta,Z_j,Y_j) = \frac{1}{n}\l( \|\vec\eps\|^2 + \| \mb Z(\theta- \theta_\ast)\|^2 - 2\langle \vec\eps,\mb Z(\theta_\ast-\theta)\rangle\r)$, where $\vec\eps =(\eps_1,\ldots,\eps_n)^T$ and $\mb Z\in \mb R^{n\times d}$ has $Z_1,\ldots,Z_n$ as rows. Cauchy-Schwarz inequality combined with a simple relation $2|ab|\leq a^2/2 + 2b^2$ that holds for all $a,b\in \mb R$ yield that 
\[
\frac{1}{n}\sum_{j=1}^n \ell(\theta,Z_j,Y_j) \geq \frac{1}{2n}\| \mb Z(\theta- \theta_\ast)\|^2  - \frac{1}{n}\|\vec\eps\|^2,
\]
hence $\inf_{\|\theta - \theta_\ast\| \geq R} \frac{1}{n}\sum_{j=1}^n \ell(\theta,Z_j,Y_j) \geq \frac{R^2}{2} \inf_{\|u\|=1} \langle \Sigma_n u,u\rangle - \frac{1}{n}\|\vec\eps\|^2$ where $\Sigma_n=\frac{1}{n}\sum_{j=1}^n Z_j Z_j^T$ is the sample covariance matrix. Since $\inf_{\|u\|=1} \langle \Sigma_n u,u\rangle \geq \lambda_{\min}(\Sigma) - \|\Sigma_n - \Sigma\| = \lambda_{\min}(\Sigma) - o_P(1)$ and $ \frac{1}{n}\|\vec\eps\|^2 = O_p(1)$, it is easy to conclude that assumption \ref{ass:4} holds.

We are ready to state the main results regarding consistency and asymptotic normality of the estimator  \eqref{eq:M-est-erm2}.
Recall the adversarial contamination framework defined in section \ref{sec:intro}. 
In all statements below, we assume that 
the sequences $\{k_j\}_{j\geq 1}$ and $\{n_j\}_{j\geq 1}$, corresponding the the number of subgroups and their cardinality respectively, are non-decreasing and converge to $\infty$ as $j\to\infty$, and that the total sample size is $N_j := k_j n_j$.
\begin{theorem}
\label{th:consistency}
Let assumptions \ref{ass:1}, \ref{ass:2}, \ref{ass:3} and \ref{ass:4} be satisfied. Suppose that the number of outliers $\m O_j$ and the scaling sequence are such that 
$\limsup\limits_{j \to\infty}\frac{\m O_j}{k_j}\leq c$ for a sufficiently small absolute constant $c>0$. 
Then the estimator $\wh\theta_{n_j,k_j}$ defined in \eqref{eq:M-est-erm2} is consistent: $\wh\theta_{n_j,k_j} \to \theta_0$ in probability as $j\to\infty$.
\end{theorem}
We remark that the contamination framework considered in Theorem \ref{th:consistency} is quite general: for instance, in the framework if linear regression, $X=(Z,Y)\in \mb R^d\times \mb R$, hence outliers can occur among both the predictor $X$ and response variable $Y$. On the other hand, many classical robust regression methods, such as Huber's regression \citep{huber2011robust}, only allow the outliers among the responses. 
The following theorem constitutes the main contribution of the paper.
\begin{theorem}
\label{th:normality-B}
Assume that the the sample is free of adversarial contamination (that is, $\kappa=0$). Let assumptions 
\ref{ass:1}, \ref{ass:2}, \ref{ass:3} and \ref{ass:4} be satisfied. Then 
\[
\sqrt{N_j}\l( \wh\theta_{n_j,k_j} - \theta_0\r) \xrightarrow{d} N\l(0, D^2(\theta_0)\r) \text{ as } j\to\infty,
\]
where $D^2(\theta_0) = \l[\pd^2_\theta L(\theta_0)\r]^{-1}\Sigma \l[\pd^2_\theta L(\theta_0)\r]^{-1}$ and $\Sigma = \mb E \l[ \pd_\theta \ell(\theta_0,X) \pd_\theta \ell(\theta_0,X)^T\r]$.
\end{theorem}
In essence, this result establishes that no loss of asymptotic efficiency occurs if the standard M-estimator based on empirical risk minimization is replaced by its robust counterpart $\wh\theta_N$.
For example, maximum likelihood estimator corresponds to the case when $\l\{ p_\theta, \ \theta\in \Theta\r\}$ is a family of probability density functions with respect to some $\sigma$-finite measure $\mu$ and $\ell(\theta,\cdot) = -\log p_\theta(\cdot)$. 
If it holds that $-\pd^2_\theta \,\mb E \log p_{\theta_0}(X) = I(\theta_0):=\mb E \l[ \pd_\theta \log p_{\theta_0}(X) \pd_\theta \log p_{\theta_0}(X)^T \r]$, then it follows that $\wh\theta_{N_j}$ is asymptotically equivalent to the maximum likelihood estimator. 

\subsection{Computational aspects.} 
\label{sec:computational}

Here, we briefly discuss some of the more practical aspects of the proposed estimators, including 
the choice of the scaling factors $\Delta_n$. Note that, while $\wh L(\theta,\theta')$ itself is defined as a minimizer of a convex function, it is not a convex-concave function itself, and the problem \eqref{eq:M-est-erm2} is not guaranteed to be convex-concave or have a unique solution. 
However, the gradient of $\wh L(\theta,\theta')$, both with respect to $\theta$ and $\theta'$, is easily computable: as $\sum_{j=1}^k \rho'\l(\sqrt{n}\,\frac{\bL_j (\theta) - \bL_j(\theta') - \wh L(\theta,\theta')}{\Delta_n}\r) = 0$, differentiating this expression yields that 
\[
\partial_\theta \wh L(\theta,\theta') =\frac{ \sum_{j=1}^k \partial_\theta \bL_j(\theta) \rho''\l(\sqrt{n}\,\frac{\bL_j (\theta) - \bL_j(\theta') - \wh L(\theta,\theta')}{\Delta_n}\r) }{ \sum_{j=1}^k \rho''\l(\sqrt{n}\,\frac{\bL_j (\theta) - \bL_j(\theta') - \wh L(\theta,\theta')}{\Delta_n}\r)}.
\] 
Due to this fact, gradient descent-ascent type methods for solving the problems closely related to \eqref{eq:M-est-erm2} have been proposed and have shown good performance in extended simulation studies; we refer the reader to \citep{lecue2017robust,mathieu2021excess} for the details. 

The problem of choosing the scaling factor for robust estimators of location has been studied since the seminal work of \cite{huber1964robust}. Here, we suggest setting $\Delta_n$ in a data-dependent way using the ``median absolute deviation'' (MAD) estimator; this idea has been suggested and numerically tested in \citep{mathieu2021excess}. We start with $\Delta_n:=\Delta_{n,0}$ being a fixed number (e.g., $\Delta_{n,0} = 1)$.  
Given an approximate solution $(\theta_t,\theta'_t)$, e.g., obtained via the gradient descent-ascent iteration, set $\widehat M(\theta_t,\theta'_t):=\mathrm{median}\l(\bL_1(\theta_t,\theta'_t),\ldots,\bL_k(\theta_t,\theta'_t) \r)$, and
\[
\mathrm{MAD}(\theta_t,\theta'_t) = \mathrm{median}\l( \l| \bL_1(\theta_t,\theta'_t) -\widehat M(\theta_t,\theta'_t)\r|, \ldots,\l| \bL_k(\theta_t,\theta'_t) - \widehat M(\theta_t,\theta'_t)\r| \r).
\] 
Finally, define $\wh \Delta_{n,t+1} := \frac{\mathrm{MAD}(\theta_t,\theta'_t)}{\Phi^{-1}(3/4)}$, where $\Phi$ is the distribution function of the standard normal law and the normalizing factor comes from the fact that for a sample from the normal distribution $N(\mu,\sigma^2)$, the expected value of $\mathrm{MAD}$ equals $\Phi^{-1}(3/4)\sigma$. The scaling factor can be updated again after a fixed number of iterations. Our theoretical results do not allow for a data-dependent choice of $\Delta_n$ however, and it would be an interesting avenue for further investigation.

\noindent We include a simple proof-of-concept numerical simulation in appendix \ref{sec:logistic}.

	
\section{Proofs.}

The proof of Theorem \ref{th:normality-B} uses characterization of $\wh\theta_{n,k}$ as the solution of the min-max problem, and follows a standard pattern of consequently establishing consistency, rate of convergence and finally the asymptotic normality. The arguments are quite general and can be extended beyond the classes that satisfy Lipschitz property imposed by assumption \ref{ass:3}. Since $\wh L(\theta_1,\theta_2)$ is defined implicitly as a solution of the convex minimization problem, we rely on the Bahadur-type linear representation of $\wh L(\theta_1,\theta_2)-L(\theta_1,\theta_2)$ with uniform control of the remainder terms.


\subsection{Technical tools.}

We recall some of the basic facts and existing results that our proofs often rely upon. Given a metric space $(T,\rho)$, the covering number $N(T,\rho,\eps)$ is defined as the smallest $N\in \mb N$ such that there exists a subset 
$F\subseteq T$ of cardinality $N$ with the property that for all $z\in T$, $\rho(z,F)\leq \eps$. 

Let $\l\{ Y(t), \ t\in T\r\}$ be a stochastic process indexed by $T$. 
We will say that it has sub-Gaussian increments with respect to some metric $\rho$ if for all $t_1,t_2\in \mb T$ and $s\in \mb R$,
\[
\mb E e^{s(Y_{t_1} - Y_{t_2})} \leq e^{\frac{s^2 \rho^2(t_1,t_2)}{2}}.
\]
\begin{fact}[Dudley's entropy bound]
\label{fact:DUD}
Let $\{Y(t), \ t\in T \}$ be a centered stochastic process with sub-Gaussian increments. Then the following inequality holds:
\be
\mb E \sup_{t\in T} |Y(t) - Y(t_0)|\leq 12\int\limits_{0}^{D(T)} \sqrt{\log N(T,\rho,\eps)}d\eps,
\ee
where $D(T)$ is the diameter of the space $T$ with respect to $\rho$.
\end{fact}
\begin{proof}
See the book by \cite{talagrand2005generic}.
\end{proof}
		
\begin{lemma}
\label{lemma:sum}
Let $\{A_n(\theta), \ \theta\in \Theta\},\{B_n(\theta), \ \theta\in \Theta\subseteq \mb R^d\}$ be sequences of stochastic processes such that for every $\theta\in \Theta$, the sequences of random variables $\{A_n(\theta)\}_{n\geq 1}$ and $\{B_n(\theta)\}_{n\geq 1}$ are stochastically bounded, and for any $\eps>0$,
\[
\limsup_{n\to\infty} \pr{\sup_{\|\theta - \theta_0\|\leq \delta} \l| A_n(\theta) - A_n(\theta_0)\r|\geq \eps} \to 0 \text{ as } \delta\to 0,
\]
\[
\limsup_{n\to\infty} \pr{\sup_{\|\theta - \theta_0\|\leq \delta} \l| B_n(\theta) - B_n(\theta_0)\r|\geq \eps} \to 0 \text{ as } \delta\to 0.
\]
Then 
\[
\limsup_{n\to\infty} \pr{\sup_{\|\theta - \theta_0\| \leq \delta} \l| A_n(\theta)B_n(\theta) - A_n(\theta_0)B_n(\theta_0)\r|\geq \eps} \to 0 \text{ as } \delta\to 0.
\]
Moreover, if there exists $c>0$ such that 
\[
\liminf_{n\to\infty} \pr{|B_n(\theta_0)|\geq c} = 1, 
\]
then the following also holds:
\[
\limsup_{n\to\infty} \pr{\sup_{\|\theta - \theta_0\| \leq \delta} \l| \frac{A_n(\theta)}{B_n(\theta)} - \frac{A_n(\theta_0)}{B_n(\theta_0)}\r|\geq \eps} \to 0 \text{ as } \delta\to 0.
\]
\end{lemma}
\begin{proof}
The result follows in a straightforward manner from the triangle inequality hence the details are omitted. 
\end{proof}	
	
The following bound that we will frequently rely upon allows one to control the error $\l| \hL(\theta) - L(\theta)\r|$ uniformly over compact subsets $\Theta'\subseteq \Theta$. 
Recall the adversarial contamination framework introduced in section \ref{sec:intro}, and define 
\[
\widetilde \Delta:=\max\l( \Delta_n, \sup_{\theta\in \Theta'}\sigma(\theta,\theta_0)\r).
\] 
\begin{lemma}
\label{lemma:unif}
Let $\m L = \{ \ell(\theta,\cdot), \ \theta \in \Theta \}$ be a class of functions mapping $S$ to $\mb R$, and assume that 
$\sup_{\theta\in \Theta'} \mb E\l| \ell(\theta,X) - \ell(\theta_0,X) - L(\theta,\theta_0)\r|^{2+\tau}<\infty$ for some $\tau\in [0,1]$. Then there exist absolute constants $c,\, C>0$ and a function $g_{\tau}(x,\theta)$ satisfying $g_{\tau}(x,\theta) \buildrel{x\to\infty}\over{=}\begin{cases} o(1), & \tau=0, 
\\
O(1), & \tau>0
\end{cases}$ such that for all $s>0,$ $n$ and $k$ satisfying
\begin{multline*}
\frac{s}{\sqrt{k}\Delta_n}\,\mb E\sup_{\theta\in \Theta'} \frac{1}{\sqrt{N}}\l| \sum_{j=1}^N \l(\ell(\theta,X_j) - \ell(\theta_0,X_j) - L(\theta,\theta_0)\r) \r| 
\\
+ \sup_{\theta \in \Theta'} \l[g_{\tau}(n,\theta) \frac{\mb E \l| \ell(\theta,X) - \ell(\theta_0,X) - L(\theta,\theta_0)\r|^{2+\tau}}{\Delta_n^{2+\tau}n^{\tau/2}} \r]
+ \frac{\m O}{k} \leq c,
\end{multline*}
the following inequality holds with probability at least $1 - \frac{1}{s}$:
\ml{
\sup_{\theta\in \Theta'}\l| 
\hL(\theta,\theta_0) - L(\theta,\theta_0) \r| \leq 
C 
\Bigg[ s\cdot\frac{\widetilde \Delta}{\Delta_n} \mb E\sup_{\theta\in \Theta'} \l|\frac{1}{N} \sum_{j=1}^N \Big( \ell(\theta,X_j)-\ell(\theta_0,X_j) - L(\theta,\theta_0) \Big) \r| 
\\
+ \widetilde \Delta \l( \frac{1}{\sqrt{n}}\frac{\m O}{k}
+ \frac{1}{\sqrt n}\sup_{\theta \in \Theta'}\l[ g_{\tau}(n,\theta) \frac{\mb E \l| \ell(\theta,X) - \ell(\theta_0,X) - L(\theta,\theta_0)\r|^{2+\tau}}{\Delta_n^{2+\tau}n^{\tau/2}}\r] \r) \Bigg].
}
\end{lemma}
\noindent We will only use the bound of the lemma with $\tau=0$. The proof of this bound is similar to the argument behind Theorem 3.1 in \citep{minsker2018uniform}; for the readers' convenience, we present the details in section \ref{sec:proof-unif}. 
For the illustration purposes, assume that $\m O=0$, whence the result above implies that as long as 
$\mb E\sup_{\theta\in \Theta'} \frac{1}{\sqrt{N}}\sum_{j=1}^N \l| \ell(\theta,X_j)) -\ell(\theta_0,X_j) - L(\theta,\theta_0) \r| = O(1)$ 
and $\sigma(\Theta')\lesssim\Delta_n = O(1)$, 
\[
\sup_{\theta\in \Theta'}\l| \hL(\theta,\theta_0) - L(\theta,\theta_0) \r| = O_p\l(N^{-1/2} + n^{-(1+\tau)/2}\Delta_n^{-(2+\tau)}\r).
\] 
Moreover, if $\m O = \kappa N$ and $\Delta_n=O(1)$, then, setting $k \asymp N \kappa^{\frac{2}{2+\tau}}$, we see that  
\[
\sup_{\theta\in \Theta'}\l| \hL(\theta,\theta_0) - L(\theta,\theta_0) \r| = O_p\l(N^{-1/2} + \kappa^{\frac{1+\tau}{2+\tau}}\r).
\]

\subsection{Proof of Theorem \ref{th:consistency}.}

To simplify and clarify the notation, we will omit subscript $j$ in most cases and simply write ``$k,n$'' instead of ``$k_j,n_j$'' to denote the increasing sequences of the number of subgroups and their cardinalities. 
For every $\theta'\in \Theta$, define $\wh \theta(\theta'):=\argmax_{\theta\in \Theta} \wh L(\theta',\theta) = \argmin_{\theta\in\Theta}\wh L(\theta,\theta')$ (here, we assume that the maximum is attained so that $\wh \theta(\theta')$ is well defined; however, the argument also holds with $\wh \theta(\theta')$ replaced by a near-maximizer). 
We will set $\wh\theta_{n,k}^{(1)}:=\wh\theta_{n,k}$ and $\wh\theta_{n,k}^{(2)}:=\wh \theta(\wh\theta_{n,k}^{(1)})$. Observe that $\wh L\l(\wh\theta_{n,k}^{(1)},\wh\theta_{n,k}^{(2)}\r) \leq \wh L\l(\theta_0, \wh\theta(\theta_0)\r)$, hence 
whenever $\|\wh \theta_{n,k}^{(j)}-\theta_0\|\leq R, \ j=1,2$, 
\ml{
L(\wh\theta_{n,k}^{(1)}) - L(\wh\theta_{n,k}^{(2)}) = L(\wh\theta_{n,k}^{(1)}) - L(\wh\theta_{n,k}^{(2)}) \pm \wh L(\wh\theta_{n,k}^{(1)},\wh\theta_{n,k}^{(2)}) 
\\
\leq \wh L\l(\theta_0, \wh\theta(\theta_0)\r) +  \sup_{\|\theta_j - \theta_0\|\leq R, j=1,2} \l|\wh L(\theta_1,\theta_2) - L(\theta_1,\theta_2)\r| 
\\
\leq L(\theta_0) - L(\wh\theta(\theta_0)) 
+ 2\sup_{\|\theta_j - \theta_0\|\leq R, j=1,2} \l|\wh L(\theta_1,\theta_2) - L(\theta_1,\theta_2) \r| 
\\
\leq 2\sup_{\|\theta_j - \theta_0\|\leq R, j=1,2} \l|\wh L(\theta_1,\theta_2) - L(\theta_1,\theta_2)\r|, 
}
where we used the fact that $L(\theta_0) - L(\wh\theta(\theta_0)) \leq 0$ in the last step. 
On the other hand, for any $\eps>0$, 
\[
\inf_{\|\theta_1 - \theta_0\|\geq \eps} \sup_{\theta_2}\l( L(\theta_1) - L(\theta_2)\r) > L(\theta_0) + \delta - L(\theta_0) = \delta
\]
where $\delta:=\delta(\eps)>0$ exists in view of assumption \ref{ass:2}. Therefore, 
\ml{
\pr{\|\wh\theta^{(1)}_{n,k} - \theta_0\|\geq \eps} \leq \pr{\sup_{\|\theta_j - \theta_0\|\leq R, j=1,2} \l|\wh L(\theta_1,\theta_2) - L(\theta_1,\theta_2) \r| >\delta/2 }
\\ 
+ \pr{\l\| \wh\theta_{n,k}^{(1)} - \theta_0\r\|>R \text{ or } \l\| \wh\theta_{n,k}^{(2)}-\theta_0\r\|>R}.
}
It follows from Lemma \ref{lemma:unif} that
\[
\sup_{\|\theta_j - \theta_0\|\leq R, j=1,2} \l|\wh L(\theta_1,\theta_2) - L(\theta_1,\theta_2)\r|  \to 0 \text{ in probability}
\]
as long as $\limsup_{k,n\to\infty}\frac{\m O(k,n)}{k}\leq c$ as $n,k\to \infty$. 
Indeed, to verify this, it suffices to show that 
\[
\limsup_{N\to\infty} \mb E\sup_{\|\theta_j - \theta_0\|\leq R, j=1,2} \l|\frac{1}{\sqrt N} \sum_{j=1}^N \l( \ell(\theta_1,X_j) - \ell(\theta_2,X_j) - L(\theta_1,\theta_2) \r) \r| <\infty,
\]
which follows from the triangle inequality and the relation
\begin{equation}
\label{eq:exp-sup}
\limsup_{N\to\infty} \mb E\sup_{\|\theta_1 - \theta_0\|\leq R} \l|\frac{1}{\sqrt N} \sum_{j=1}^N \l( \ell(\theta_1,X_j) - \ell(\theta_0,X_j) - L(\theta_1,\theta_0) \r) \r| <\infty.
\end{equation}
To establish the latter, we use a well-known argument based on symmetrization inequality and Dudley's entropy integral bound (see Fact \ref{fact:DUD}). 
Let $\eps_1,\ldots,\eps_N$ be i.i.d. random signs, independent of the data $X_1,\ldots,X_N$. Then symmetrization inequality \citep{wellner1} yields that
\ml{
\mb E\sup_{\theta\in \Theta:\|\theta - \theta_0\|\leq R} \frac{1}{\sqrt N}\l| \sum_{j=1}^N \l( \ell(\theta,X_j) - \ell(\theta_0,X_j)- L(\theta,\theta_0) \r) \r|
\\
\leq 2\mb E\sup_{\theta\in \Theta:\|\theta - \theta_0\|\leq R} \frac{1}{\sqrt N}\l| \sum_{j=1}^N \eps_j \l(\ell(\theta,X_j) - \ell(\theta_0,X_j)\r) \r|.
}
Conditionally on $X_1,\ldots,X_N$, the process $\ell(\theta,\cdot)\mapsto \frac{1}{\sqrt N} \sum_{j=1}^N \eps_j \l(\ell(\theta,X_j) - \ell(\theta_0,X_j)\r)$ has sub-Gaussian increments with respect to the semi-metric $d_N^2(\theta_1,\theta_2):=\frac{1}{N}\sum_{j=1}^N \l( \ell(\theta_1,X_j) - \ell(\theta_2,X_j)\r)^2$. 
It follows from compactness of the set $B(\theta_0,R)=\{\theta: \, \|\theta - \theta_0\|\leq R\}$ and assumption \ref{ass:3} that there exist $\theta_1,\ldots,\theta_{N(R)}$ such that  $\bigcup_{j=1}^{N(R)} B(\theta_j, r(\theta_j)) \supseteq B(\theta_0,R)$ and 
\[
|\ell(\theta',x) - \ell(\theta'',x) |\leq \m V(x; r(\theta_j)) \|\theta' - \theta''\|
\] 
for all $\theta',\theta''\in B(\theta_j,r(\theta_j))$.
To cover $B(\theta_0,R)$ by the balls of $d_N$-radius $\tau$, it suffices to cover each of the $N(R)$ balls $B(\theta_j,r(\theta_j))$. 
It is easy to see that the latter requires at most $\l( \frac{6 r(\theta_j)\| \m V(\cdot; r(\theta_j)) \|_{L_2(P_N)}}{\tau } \r)^{d}$ balls of radius $\tau$. Therefore, 
\[
\log^{1/2} N(B(\theta_0,R),d_N,\tau)\leq  \log^{1/2}\l( \sum_{j=1}^{N(R)} \l[ \l( \frac{6 r(\theta_j)\| \m V(\cdot; r(\theta_j)) \|_{L_2(P_N)}}{\tau } \r)^{d} \vee 1\r]\r).
\] 
Note that for any $x_1,\ldots,x_m\geq 1$, $\sum_{j=1}^m x_j \leq m\prod_{j=1}^m x_j$, or 
$\log \l(\sum_{j=1}^m x_j \r) \leq \log m + \sum_{j=1}^m \log x_j$, so that 
\ml{
\log^{1/2}\l( \sum_{j=1}^{N(R)} \l[\l( \frac{6 r(\theta_j)\| \m V(\cdot; r(\theta_j)) \|_{L_2(P_N)}}{\tau } \r)^{d} \vee 1\r]\r)
\\
\leq \log^{1/2} N(R) + \sum_{j=1}^{N(R)} \sqrt{d}\log_+^{1/2}\l( \frac{6 r(\theta_j)\| \m V(\cdot; r(\theta_j))\|_{L_2(P_N)}}{\tau }\r),
}
where $\log_+(x) := \max(\log x,0)$.
Moreover, the diameter $D_N$ of the set $B(\theta_0,R)$ is at most $2\sum_{j=1}^{N(R)} r(\theta_j) \|  \m V(\cdot; r(\theta_j)) \|_{L_2(P_N)}$. Therefore, 
\ml{
\int_0^{D_N} \log^{1/2} N(B(\theta_0,R),d_N,\tau)d\tau 
\\
\leq C\l( D_N \log^{1/2} N(R) + \sqrt{d}\sum_{j=1}^{N(R)} r(\theta_j) \| V(\cdot; r(\theta_j)) \|_{L_2(P_N)} \int_0^1 \log^{1/2} (1/\tau)d\tau\r)
}
and 
\[
\mb E\sup_{\theta\in \Theta:\|\theta - \theta_0\|\leq R} \frac{1}{\sqrt N}\l| \sum_{j=1}^N \eps_j (\ell(\theta,X_j) - \ell(\theta_0,X_j))\r| 
\leq 
C \log^{1/2}(N(R))\sum_{j=1}^{N(R)} r(\theta_j) \| \m V(\cdot; r(\theta_j))\|_{L_2(P)}<\infty.
\]
It remains to establish that $\pr{\l\| \wh\theta_{n,k}^{(1)} - \theta_0\r\|>R \text{ or } \l\| \wh\theta_{n,k}^{(2)}-\theta_0\r\|>R} \to 0$. To this end, notice that by the definition of $\wh\theta_{n,k}^{(1)}$,
\ml{
0\leq\wh L(\wh\theta_{n,k}^{(1)},\wh\theta_{n,k}^{(2)}) \leq \wh L\l(\theta_0,\wh\theta(\theta_0)\r) \leq \underbrace{L(\theta_0) - L\l( \wh\theta(\theta_0)\r)}_{\leq 0} + \sup_{\|\theta - \theta_0\|\leq R}\l|\wh L(\theta_0,\theta) - L(\theta_0,\theta)\r|
}
on the event $\l\{ \|\wh\theta(\theta_0) - \theta_0\| \leq R\r\}$. It has already been established that 
\[
\sup_{\|\theta - \theta_0\|\leq R}\l|\wh L(\theta_0,\theta) - L(\theta_0,\theta)\r|\to 0 \text{ in probability.}
\] 
To show that $\pr{\|\wh\theta(\theta_0) - \theta_0\| > R} \to 0$ for $R$ large enough and as $n,k\to \infty$, 
recall that
\[
B(n,R,t) = \pr{\inf_{\|\theta - \theta_0\|\geq R} \frac{1}{n}\sum_{j=1}^n \ell(\theta,X_j) < L(\theta_0) + t }
\] 
and that $\lim_{R\to\infty} \limsup_{n\to\infty} B(n,R,t) = 0$ for some $t>0$ in view of Assumption \ref{ass:4}. As moreover $\frac{1}{n}\sum_{j=1}^n \ell(\theta_0,X_j)\to L(\theta_0)$ in probability, one can choose $R_0$ and $n_0$ such that 
\[
\tilde B(n,R,t)=\pr{\inf_{\|\theta - \theta_0\|\geq R} \frac{1}{n}\sum_{j=1}^n \ell(\theta,X_j) - \frac{1}{n}\sum_{j=1}^n \ell(\theta_0,X_j) < t/2 }<0.01
\] 
for all $n\geq n_0$ and $R\geq R_0$. As $\hL(\theta,\theta_0) = \argmin_{z \in \R} \sum_{j = 1}^k \rho\left(\frac{\sqrt{n}}{\Delta_n}(\bar L_{j}(\theta) - \bar L_{j}(\theta_0) - z)\right)$, 
it solves the equation $\sum_{j = 1}^k \rho'\left(\frac{\sqrt{n}}{\Delta_n}(\bar{L}_{j}(\theta) - \bar L_{j}(\theta_0) - \hL(\theta,\theta_0))\right) = 0$. 
Assumption \ref{ass:1} implies that $\rho'(x) = \|\rho'\|_\infty$ for $x \geq 2$. 
Therefore, $\hL(\theta,\theta_0) < t/4$ only if $\bar L_j(\theta) - \bar L_{j}(\theta_0) < t/4+ 2\frac{\Delta_n}{\sqrt n}$ for $j\in J$ such that $|J|\geq k/2$. 
To see this, suppose that there exists a subset $J'\subseteq \{1,\ldots,k\}$ of cardinality $|J'|>k/2$ such that $\bar L_j(\theta) - \bar L_{j}(\theta_0) \geq t/4 + 2\frac{\Delta_n}{\sqrt n}$ for $j\in J'$ while $\hL(\theta,\theta_0) < t/4$. 
In turn, it implies that 
$\bar L_j(\theta) - \bar L_{j}(\theta_0) >  2\frac{\Delta_n}{\sqrt n}, \ j\in J'$, whence 
\[
\sum_{j = 1}^k \rho'\left(\frac{\sqrt{n}}{\Delta_n}(L_{j}(\theta) - \bar L_{j}(\theta_0) - \hL(\theta,\theta_0))\right) > \frac{k}{2}\|\rho\|_\infty + \sum_{j\notin J'} \rho'\left(\frac{\sqrt{n}}{\Delta_n}(L_{j}(\theta) - \bar L_{j}(\theta_0) - \hL(\theta,\theta_0))\right) > 0,
\] 
leading to a contradiction. Therefore, 
\mln{
\label{eq:mom}
\pr{\inf_{\|\theta - \theta_0\|\geq R} \hL(\theta,\theta_0) <  t/4} 
\\
\leq \pr{\exists J\subseteq \{1,\ldots,k\}, \ |J|\geq k/2: \ \inf_{\|\theta - \theta_0\|\geq R} \bar L_j(\theta) -\bar L_{j}(\theta_0) < t/4 + 2\frac{\Delta_n}{\sqrt n}, \ j\in J} 
\\
\leq {k\choose \lfloor k/2 \rfloor} \l( \tilde B(n,R,t) \r)^{k/2} \leq (2e)^{k/2}  \l(\tilde B(n,R,t)\r)^{k/2}
}
whenever $2\frac{\Delta_n}{\sqrt n} \leq t/2$ and where we used the inequality ${M \choose l}\leq \l( Me/l\r)^l$. 
Moreover, if $n\geq n_0$ and $R\geq R_0$, we deduce that $(2e)^{k/2}  \l(\tilde B(n,R,t)\r)^{k/2} < 0.25^k \to 0$ as $k\to\infty$. 
As $\hL(\theta_0,\theta_0) \equiv 0$, the preceding display implies that 
$\pr{\|\wh \theta(\theta_0) - \theta_0\| < R}\to 1$ as $n,k,R\to\infty$. We have thus shown that
\ben
\label{eq:rhs}
\wh L(\wh\theta_{n,k}^{(1)},\wh\theta_{n,k}^{(2)}) \to 0 \text{ in probability. }
\een
On the other hand, by the definition of $\wh\theta_{n,k}^{(2)}$, it holds that $\wh L(\wh\theta_{n,k}^{(1)},\wh\theta_{n,k}^{(2)}) \geq \wh L(\wh\theta_{n,k}^{(1)},\theta_0)$. 
Now, assume that $\| \wh\theta_{n,k}^{(1)} - \theta_0\| > R$ while $\wh L(\wh\theta_{n,k}^{(1)},\theta_0) < L(\theta_0) + t/2 - L(\theta_0) = t/2$. Arguing as before, we see that there exists $J'\subset \{1,\ldots,k\}$ such that $|J'|>k/2$ and 
$\bar L_j(\wh\theta_{n,k}^{(1)}) - \bar L_j(\theta_0) < L(\theta_0) + t/2 - L(\theta_0) + 2\frac{\Delta_n}{\sqrt n}$ for $j\in J'$, which implies the inequalities
\[
\inf_{\|\theta -\theta_0\|>R} \bar L_j(\theta) < L(\theta_0) + t/2 + 2\frac{\Delta_n}{\sqrt n} + \l( \bar L_j(\theta_0)-L(\theta_0)\r), \ j\in J'.
\] 
Clearly, $\pr{\l| \l( \bar L_j(\theta_0)-L(\theta_0)\r)\r| \geq t/4}\leq \frac{16}{n t^2}\var\l( \ell(\theta_0,X)\r)$, therefore, for $n$ and $R$ large enough, $\pr{\inf_{\|\theta -\theta_0\|>R} \bar L_j(\theta) < L(\theta_0) + t/2 + 2\frac{\Delta_n}{\sqrt n} + \l( \bar L_j(\theta_0)-L(\theta_0)\r)}<0.01$ for any $j$. 
Reasoning as in \eqref{eq:mom}, we see that 
\[
\pr{\wh L(\wh\theta_{n,k}^{(1)},\theta_0) < t/2 \text{ and } \| \wh\theta_{n,k}^{(1)} - \theta_0\| > R} \to 0 \text{ as } k,n\to\infty.
\]
We deduce that on the one hand,
\[
\pr{\wh L(\wh\theta_{n,k}^{(1)},\theta_0) \geq t/2 \bigcap \| \wh\theta_{n,k}^{(1)} - \theta_0\| > R} \to \pr{ \| \wh\theta_{n,k}^{(1)} - \theta_0\| > R}.
\]
In view of \eqref{eq:rhs}, we see that on the other hand,
\[
\pr{\wh L(\wh\theta_{n,k}^{(1)},\theta_0) \geq t/2 \bigcap \| \wh\theta_{n,k}^{(1)} - \theta_0\| > R} \leq \pr{\wh L(\wh\theta_{n,k}^{(1)},\theta_0) \geq t/2} \to 0,
\]
implying that $\pr{\| \wh\theta_{n,k}^{(1)} - \theta_0\| > R} \to 0$ for $R$ large enough as $n,k\to\infty$. 

Finally, assume that $\| \wh\theta_{n,k}^{(2)} - \theta_0\| > R$ and that $\wh L(\wh\theta_{n,k}^{(1)},\wh\theta_{n,k}^{(2)})>L(\wh\theta_{n,k}^{(1)}) - L(\theta_0) - t/2$. 
Repeating the reasoning behind \eqref{eq:mom}, we see that the latter implies that there exists $J'\subset \{1,\ldots,k\}$ such that $|J'|>k/2$ and 
$\bar L_j(\wh\theta_{n,k}^{(1)}) - \bar L_j(\wh\theta_{n,k}^{(2)}) > L(\wh\theta_{n,k}^{(1)}) - \l( L(\theta_0) + t/2 + 2\frac{\Delta_n}{\sqrt n} \r)$ for $j\in J'$, yielding that on the event $\l\{ \| \wh\theta_{n,k}^{(1)} - \theta_0\|\leq R\r\}$,
\ml{
\inf_{\|\theta - \theta_0\|>R} \bar L_j(\theta) < L(\theta_0) + t/2 + 2\frac{\Delta_n}{\sqrt n} + \l(\bar L_j(\wh\theta_{n,k}^{(1)}) - L(\wh\theta_{n,k}^{(1)})\r) 
\\
\leq L(\theta_0) + t/2 + 2\frac{\Delta_n}{\sqrt n} + \sup_{\|\theta'-\theta_0\|\leq R} \l| \bar L_j(\theta') - L(\theta')\r|
}
for $j\in J'$. We have shown before that $\pr{ \| \wh\theta_{n,k}^{(1)} - \theta_0\|>R}\to 0$ for $R$ large enough as $n,k\to\infty$. As $\mb E \sup_{\|\theta'-\theta_0\|\leq R} \l| \bar L_j(\theta') - L(\theta')\r| \to 0$ for any $R>0$ as $n\to\infty$ (indeed, this follows from \eqref{eq:exp-sup} and the triangle inequality), for $n$ and $R$ large enough, the argument similar to \eqref{eq:mom} implies that
\[
\pr{\sup_{\|\theta - \theta_0\| > R} \hL(\wh\theta_{n,k}^{(1)},\theta) >  L(\wh\theta_{n,k}^{(1)}) - L(\theta_0) - t/2}\to 0 \text{ as } k\to\infty,
\]
therefore $\pr{\| \wh\theta_{n,k}^{(2)} - \theta_0\| > R \,\bigcap\, \wh L\l( \wh\theta_{n,k}^{(1)}, \wh\theta_{n,k}^{(2)}\r) \leq L( \wh\theta_{n,k}^{(1)}) - (L(\theta_0) + t/2) } \to \pr{\| \wh\theta_{n,k}^{(2)} - \theta_0\| > R}$.
On the other hand, 
\ml{
\pr{ \wh L\l( \wh\theta_{n,k}^{(1)}, \wh\theta_{n,k}^{(2)}\r) \leq L( \wh\theta_{n,k}^{(1)}) - (L(\theta_0) + t/2)} \leq 
\pr{  \wh L\l( \wh\theta_{n,k}^{(1)}, \theta_0\r) \leq L( \wh\theta_{n,k}^{(1)}) - (L(\theta_0) + t/2)}
\\
\pr{L(\wh\theta_{n,k}^{(1)}) - L(\theta_0) - \sup_{\| \theta - \theta_0\|\leq R} \l| \wh L(\theta,\theta_0) - (L(\theta)-L(\theta_0))\r| \leq L( \wh\theta_{n,k}^{(1)}) - (L(\theta_0) + t/2)} 
\\
+\pr{\|\wh\theta_{n,k}^{(1)} - \theta_0\|>R}
= \pr{\sup_{\| \theta - \theta_0\|\leq R} \l| \wh L(\theta,\theta_0) - (L(\theta)-L(\theta_0))\r|\geq t/2} + \pr{\|\wh\theta_{n,k}^{(1)} - \theta_0\|>R} \to 0
}
for $R$ large enough as $n,k\to\infty$, therefore completing the proof of consistency.

\subsection{Proof of Theorem \ref{th:normality-B}.}
\label{proof:normality}

As in the proof of Theorem \ref{th:consistency}, we will omit subscript $j$ and write ``$k,n$'' instead of ``$k_j,n_j$'' to denote the increasing sequences of the number of subgroups and their cardinalities. 
The argument is divided into two steps. 
The first step consists in establishing the fact that the estimator $\wh\theta_{n,k}$ converges to $\theta_0$ at $\sqrt N$-rate, while on the second step we prove asymptotic normality by ``zooming'' to the resolution level $N^{-1/2}$; this proof pattern is quite standard in the empirical process theory \citep{wellner1}. 

As in the proof of Theorem \ref{th:consistency}, we set  $\wh \theta(\theta'):=\argmax_{\theta\in \Theta} \wh L(\theta',\theta) = \argmin_{\theta\in\Theta}\wh L(\theta,\theta')$ and define $\wh\theta_{n,k}^{(1)}:=\wh\theta_{n,k}$ and $\wh\theta_{n,k}^{(2)}:=\wh \theta(\wh\theta_{n,k}^{(1)})$.
We present a detailed argument establishing the convergence rate for $\wh\theta_{n,k}^{(1)}$, and outline the modifications necessary to establish the result for $\wh\theta_{n,k}^{(2)}$. 
Our goal is to show that 
\ben
\label{eq:theta_N^1}
\lim_{M\to\infty}\limsup_{n,k\to\infty}\pr{\sqrt{N}\| \wh\theta_{n,k}^{(1)} - \theta_0\| \geq 2^M} = 0.
\een 
Define $S_{N,j}:=\l\{ \theta: \ 2^{j-1}/\sqrt{N}<\|\theta - \theta_0\| \leq 2^j/\sqrt{N} \r\}$, $\bar S_{N,j}:=\l\{ \theta: \ 0\leq \|\theta - \theta_0\| \leq 2^j/\sqrt{N} \r\}$, and observe that 
\[
\sqrt{N}\| \wh\theta_{n,k}^{(1)} - \theta_0\| \geq 2^M \implies \inf_{\theta\in S_{N,j}} \l( \wh L(\theta,\wh\theta(\theta)) - \wh L(\theta_0,\wh\theta(\theta_0))\r) \leq 0 \text{ for some } j>M, 
\] 
where $\wh \theta(\theta'):=\argmax_{\theta\in \Theta} \wh L(\theta',\theta)$. As $\wh L(\theta,\wh\theta(\theta))\geq \wh L(\theta,\theta_0)$ for any $\theta$,  the inequality $\sqrt{N}\| \wh\theta^{(1)} - \theta_0\| \geq 2^M$ implies that
$\inf_{\theta\in S_{N,j}} \l( \wh L(\theta,\theta_0) - \wh L(\theta_0,\wh\theta(\theta_0)\r) \leq 0 \text{ for some } j>M,$
which in turn entails that 
\[
\inf_{\theta\in S_{N,j}} \l( \wh L(\theta,\theta_0) - L(\theta,\theta_0) - \wh L(\theta_0,\wh\theta(\theta_0)) + L(\theta_0,\wh\theta(\theta_0))\r) \leq L(\theta_0,\wh\theta(\theta_0)) - \inf_{\theta\in S_{N,j}} L(\theta,\theta_0)
\]
for some $j>M$. Since $L(\theta_0,\wh\theta(\theta_0)) - \inf_{\theta \in S_{N,j}} L(\theta,\theta_0) \leq 0$ by the definition of $\theta_0$, 
the previous display yields that
\ml{
\sup_{\theta\in S_{N,j}} \l| \wh L(\theta,\theta_0) - L(\theta,\theta_0) - \wh L(\theta_0,\wh\theta(\theta_0)) + L(\theta_0,\wh\theta(\theta_0))\r| 
\\
\geq  \inf_{\theta\in S_{N,j}} L(\theta,\theta_0) - L(\theta_0,\wh\theta(\theta_0)) \geq  \inf_{\theta\in S_{N,j}} L(\theta,\theta_0),
}
which further implies that either 
\[
\sup_{\theta\in S_{N,j}} \l| \wh L(\theta,\theta_0) - L(\theta,\theta_0)\r| \geq  \inf_{\theta\in S_{N,j}} \frac{L(\theta,\theta_0)}{2}, \text{ or } 
\l| \wh L(\theta_0,\wh\theta(\theta_0)) - L(\theta_0,\wh\theta(\theta_0))\r| \geq  \inf_{\theta\in S_{N,j}} \frac{L(\theta,\theta_0)}{2}.
\]
Let $0<\eta_1\leq r(\theta_0)$ be small enough so that $L(\theta) - L(\theta_0) \geq c \|\theta - \theta_0\|^2$ for $\theta$ such that $\|\theta-\theta_0\| \leq \eta_1$ (existence of $\eta_1$ follows from assumption \ref{ass:2}), and observe that $\pr{\| \wh\theta_{n,k}^{(1)} - \theta_0 \|\geq \eta_1}\to 0$ as $n,k\to\infty$ due to consistency of the estimator. We then have
\mln{
\label{eq:master-bound}
\pr{\sqrt{N}\| \wh\theta_{n,k}^{(1)} - \theta_0\| \geq 2^M} \leq \pr{\sqrt N\l| \wh L(\theta_0,\wh\theta(\theta_0)) - L(\theta_0,\wh\theta(\theta_0))\r| \geq c \frac{2^{2M}}{\sqrt N}}
\\
+ \pr{\bigcup_{j: j\geq M+1, \ \frac{2^j}{\sqrt N}\leq \eta_1}\sup_{\theta\in S_{N,j}} \sqrt{N}\l| \wh L(\theta,\theta_0) - L(\theta,\theta_0)\r| \geq c \frac{2^{2j-2}}{\sqrt N}} + \pr{\| \wh\theta_{n,k}^{(1)} - \theta_0\| \geq \eta_1}.
}
To estimate $\pr{\bigcup_{j: j\geq M+1, \ \frac{2^j}{\sqrt N}\leq \eta_1}\sup_{\theta\in S_{N,j}} \sqrt{N}\l| \wh L(\theta,\theta_0) - L(\theta,\theta_0)\r| \geq c \frac{2^{2j-2}}{\sqrt N}}$, we invoke Proposition \ref{lemma:b-k} applied to the class $\l\{ \ell(\theta,\cdot) - \ell(\theta_0,\cdot), \ \theta\in \bar S_{N,j}\r\}$. 
Since $|\ell(\theta,x) - \ell(\theta_0,x)| \leq V(x; r(\theta_0))\frac{2^j}{\sqrt N}$ by assumption \ref{ass:3}, it is easy to see that $\sigma^2(\delta)\leq \mb EV^2(x; r(\theta_0)) \frac{2^{2j}}{N}$. Together with the union bound applied over $M<j\leq J_{\max}:= \lfloor \log(\sqrt{N}\eta_1)\rfloor+1$ with $s_j:=j^2$, it implies that for all $\theta\in S_{N,j}$, $M+1\leq j\leq J_{\max}$,
\mln{
\label{eq:b-k-2}
\sqrt{N}\l( \wh L(\theta,\theta_0) - L(\theta,\theta_0) \r) 
 \\
 = \frac{\Delta_n}{\mb E \rho'' \l( \frac{\sqrt{n}}{\Delta_n} \l(\bar L_1(\theta) - \bar L_1(\theta_0) - L(\theta,\theta_0) \r) \r)}
\frac{1}{\sqrt{k}}\sum_{i=1}^k \rho' \l( \frac{\sqrt{n}}{\Delta_n}\l(\bar L_i(\theta) - \bar L_i(\theta_0) - L(\theta,\theta_0) \r)\r)  
\\
+\m R_{n,k,j}(\theta),
}
where 
\[
\sup_{\theta\in \bar S_{N,j}}\l| \m R_{n,k,j}(\theta) \r| \leq C(d,\theta_0) \l(\frac{2^{2j}}{N} \frac{j^4}{\sqrt k}  + \frac{2^{3j} j^2}{N^{3/2}} + \sqrt{k}\frac{2^{6j}}{N^3}\r)
\] 
uniformly over all $M\leq j\leq J_{\max}$ with probability at least $1 - 3  \sum_{j: j\geq M+1}  j^{-2} \geq 1 - \frac{C}{M}$. 
Let $\m E$ denote the event of probability at least $1 - \frac{C}{M}$ on which the previous representation holds. 
Moreover, observe that, in view of Lemma \ref{lemma:lindeberg}, for $\eta_1$ small enough and $N$ large enough,
\[
\sup_{\|\theta - \theta_0\|\leq \eta_1}\l| \mb E \rho'' \l( \frac{\sqrt{n}}{\Delta_n} \l(\bar L_1(\theta) - \bar L_1(\theta_0) - L(\theta,\theta_0) \r) \r) - \rho''(0)\r| \leq \frac{\rho''(0)}{2} = \frac{1}{2}.
\]
Taking this fact into account and noting that $\frac{2^{j}}{\sqrt N} \frac{j^4}{\sqrt k}  + \frac{2^{2j} j^2}{N} + \sqrt{k}\frac{2^{5j}}{N^{5/2}} \leq c 2^j$ for any $j\leq J_{\max}$ and any $c>0$ given that $n$ is large enough (in particular, it implies that the remainder term $\m R_{n,k,j}(\theta)$ is smaller than $\frac{c}{2}\frac{2^{2j-2}}{\sqrt N}$ on event $\m E$), we deduce that 
\ml{
\pr{\bigcup_{j: j\geq M+1, \ \frac{2^j}{\sqrt N}\leq \eta_1}\sup_{\theta\in S_{N,j}} \sqrt{N}\l| \wh L(\theta,\theta_0) - L(\theta,\theta_0)\r| \geq c \frac{2^{2j-2}}{\sqrt N}} \leq \frac{C}{M}
\\
+ \sum_{j: j\geq M+1, \ \frac{2^j}{\sqrt N}\leq \eta_1} \pr{\sup_{\theta\in S_{N,j}} 
\l| \frac{1}{\sqrt{k}}\sum_{i=1}^k \rho' \l( \frac{\sqrt{n}}{\Delta_n}\l(\bar L_i(\theta) - \bar L_i(\theta_0) - L(\theta,\theta_0) \r)\r) \r| \geq c_1 \frac{2^{2j}}{\sqrt N}}.
}
Invoking Lemma \ref{lemma:lindeberg} again, we see that 
\[
\sup_{\theta\in \bar S_{N,j}}\l| \mb E \rho' \l( \frac{\sqrt{n}}{\Delta_n}\l(\bar L_1(\theta) - \bar L_1(\theta_0) - L(\theta,\theta_0) \r)\r) \r|
\leq C \frac{2^{2j}}{N}.
\]
Let us denote $\rho'_{n,i}(\theta,\theta_0) =  \rho' \l( \frac{\sqrt{n}}{\Delta_n}\l(\bar L_i(\theta) - \bar L_i(\theta_0) - L(\theta,\theta_0) \r)\r), \ i=1,\ldots,k$ for brevity. As $\sqrt{k}\frac{2^{2j}}{N} \leq c' \frac{2^{2j}}{\sqrt N}$ for any $c'>0$ and sufficiently large $n$, 
\ml{
 \pr{\sup_{\theta\in S_{N,j}} 
\l| \frac{1}{\sqrt{k}}\sum_{i=1}^k \rho'_{n,i}(\theta,\theta_0)\r| \geq c_1 \frac{2^{2j}}{\sqrt N}} 
\\
\leq  \pr{ \sup_{\theta\in S_{N,j}} 
\Bigg| \frac{1}{\sqrt{k}}\sum_{i=1}^k \Bigg(\rho'_{n,i}(\theta,\theta_0) - \mb E\rho'_{n,i}(\theta,\theta_0) \Bigg) \Bigg| \geq c_2 \frac{2^{2j}}{\sqrt N} }
\\
\leq \frac{\sqrt{N}}{c_2 2^{2j}} \mb E\sup_{\theta\in S_{N,j}} 
\Bigg| \frac{1}{\sqrt{k}}\sum_{i=1}^k \Bigg(\rho'_{n,i}(\theta,\theta_0) - \mb E\rho'_{n,i}(\theta,\theta_0) \Bigg) \Bigg|
}
where we used Markov's inequality on the last step. 
To bound the expected supremum, we proceed in exactly the same fashion (using symmetrization, contraction and desymmetrization inequalities) as in the proof of Proposition \ref{lemma:denom}, and deduce that 
\ml{
\mb E\sup_{\theta\in S_{N,j}} 
\Bigg| \frac{1}{\sqrt{k}}\sum_{i=1}^k \Bigg(\rho'_{n,i}(\theta,\theta_0) - \mb E\rho'_{n,i}(\theta,\theta_0) \Bigg) \Bigg|
\\
\leq \frac{C}{\Delta_n}\mb E\sup_{\theta\in S_{N,j}} \l| \frac{1}{\sqrt{N}}\sum_{j=1}^N \l( \ell(\theta,X_j) - \ell(\theta_0,X_j) - L(\theta,\theta_0 \r) \r|.
}
The right side of the display above can be bounded by $\frac{C(d,\theta_0)}{\Delta_n}\frac{2^j}{\sqrt N}$ (using Lemma \ref{lemma:sup-power}), implying that 
\[
\pr{ \sup_{\theta\in S_{N,j}} 
\Bigg| \frac{1}{\sqrt{k}}\sum_{i=1}^k \Bigg(\rho'_{n,i}(\theta,\theta_0) - \mb E\rho'_{n,i}(\theta,\theta_0) \Bigg) \Bigg| \geq c_2 \frac{2^{2j}}{\sqrt N} } 
\leq  \frac{C(d,\theta_0)}{\Delta_n}\frac{1}{2^j},
\]
whence 
\ml{
\pr{\bigcup_{j: j\geq M+1, \ \frac{2^j}{\sqrt N}\leq \eta_1}\sup_{\theta\in S_{N,j}} \sqrt{N}\l| \wh L(\theta,\theta_0) - L(\theta,\theta_0)\r| \geq c \frac{2^{2j-2}}{\sqrt N}} 
\\
\leq \frac{C}{M} + \frac{C(d,\theta_0)}{\Delta_n}\sum_{j\geq M} 2^{-j} \leq \frac{C}{M} + \frac{C(d,\theta_0)}{\Delta_n}2^{-M+1} \to 0 \text{ as } M\to\infty
}
whenever $n,k$ are large enough. 
In view of \eqref{eq:master-bound}, it only remains to show that 
\ben
\label{eq:d00}
\pr{\sqrt{N} \l| \wh L(\theta_0,\wh\theta(\theta_0)) - L(\theta_0,\wh\theta(\theta_0))\r| \geq c \frac{2^{2M}}{\sqrt N}} \to 0
\text{ as } n,k\to\infty.
\een
To this end, it suffices to repeat the argument presented above, with several simplifications. 
First, we will start by proving that $\lim_{M\to\infty}\limsup_{n,k\to\infty}\pr{\sqrt{N}\| \wh\theta(\theta_0) - \theta_0\| \geq 2^M} = 0$. 
We have already shown in the course of the proof of Theorem \ref{th:consistency} that $\wh\theta(\theta_0)$ is a consistent estimator of $\theta_0$, so that $\pr{\| \wh\theta(\theta_0) - \theta_0 \| \geq \eta_2} \to 0$ for any $\eta_2>0$. 
If $\sqrt{N}\| \wh\theta(\theta_0) - \theta_0\| \geq 2^M$, then $\wh\theta(\theta_0)\in S_{N,j}$ for some $j>M$, implying that 
$\sup_{\theta\in S_{N,j}} \wh L(\theta_0,\theta) \geq \wh L(\theta_0,\theta_0) = 0$, which entails the inequality 
$\sup_{\theta\in S_{N,j}} \l(\wh L(\theta_0,\theta) - L(\theta_0,\theta) \r) \geq - \sup_{\theta\in S_{N,j}} L(\theta_0,\theta) = 
 \inf_{\theta\in S_{N,j}} L(\theta,\theta_0)\geq c \frac{2^{2j-2}}{N}$ whenever $2^j/\sqrt{N}\leq \eta_2$ and $\eta_2$ is small enough. 
Therefore,
\ml{
\pr{\sqrt{N}\| \wh\theta(\theta_0) - \theta_0\| \geq 2^M} 
\leq \pr{\| \wh\theta(\theta_0) - \theta_0 \| \geq \eta_2}
\\
+ \pr{\bigcup_{j: j\geq M+1, \ \frac{2^j}{\sqrt N}\leq \eta_2}\sup_{\theta\in S_{N,j}} \sqrt{N}\l| \wh L(\theta_0,\theta) - L(\theta_0,\theta)\r| \geq c \frac{2^{2j-2}}{\sqrt N}}.
}
The probability of the union is estimated as before using Proposition \ref{lemma:b-k}, implying that it converges to $0$ as $M\to\infty$. 
To complete the proof of \eqref{eq:d00}, observe that 
\ml{
\pr{\sqrt{N}\l| \wh L(\theta_0,\wh\theta(\theta_0)) - L(\theta_0,\wh\theta(\theta_0))\r| > c \frac{2^{2M}}{\sqrt N}} 
\leq \pr{\| \wh\theta(\theta_0) - \theta_0 \| \geq \frac{2^{M}}{\sqrt N}} 
\\
+ \pr{\sup_{\|\theta - \theta_0\|\leq \frac{2^{M}}{\sqrt N}} \sqrt{N}\l| \wh L(\theta_0,\theta) - L(\theta_0,\theta)\r| \geq 
 c \frac{2^{2M}}{\sqrt N} }
}
and that $ \pr{\sup_{\|\theta - \theta_0\|\leq \frac{2^{M}}{\sqrt N}} \sqrt{N}\l| \wh L(\theta_0,\theta) - L(\theta_0,\theta)\r| \geq 
 c \frac{2^{2M}}{\sqrt N} }\leq \frac{C}{M} + \frac{C(d,\theta_0)}{\Delta_n}2^{-M} \to 0$ as $M\to \infty$, which follows from the representation \eqref{eq:b-k-2} in the same fashion as before. This completes the proof of relation \eqref{eq:theta_N^1}. 
To establish that 
\[
\lim_{M\to\infty}\limsup_{n,k\to\infty}\pr{\sqrt{N}\| \wh\theta_{n,k}^{(2)} - \theta_0\| \geq 2^M} = 0,
\]
we begin by observing that the inequality $\sqrt{N}\| \wh\theta_{n,k}^{(2)} - \theta_0\| \geq 2^M$ 
implies that $\sup_{\theta\in S_{N,j}} \wh L(\wh\theta_{n,k}^{(1)},\theta) \geq \wh L(\wh\theta_{n,k}^{(1)},\theta_0)$ for some $j>M$. If $\frac{2^j}{\sqrt{N}}\leq \eta_3$ for sufficiently small constant $\eta_3>0$, we see that it further entails the inequality
\ml{
\sup_{\theta\in S_{N,j}} \l( \wh L(\wh\theta_{n,k}^{(1)},\theta) - L(\wh\theta_{n,k}^{(1)},\theta) - \wh L(\wh\theta_{n,k}^{(1)},\theta_0) + L(\wh\theta_{n,k}^{(1)},\theta_0)\r) 
\\
\geq -\sup_{\theta\in S_{N,j}} L(\wh\theta_{n,k}^{(1)},\theta) + L(\wh\theta_{n,k}^{(1)},\theta_0)
= \inf_{\theta\in S_{N,j}} L(\theta,\wh\theta_{n,k}^{(1)}) +  L(\wh\theta_{n,k}^{(1)},\theta_0) 
\\
= \inf_{\theta\in S_{N,j}} L(\theta,\theta_0) \geq c \frac{2^{2j-2}}{N}.
}
We deduce from the display above that 
\ml{
\pr{\sqrt{N}\| \wh\theta_{n,k}^{(2)} - \theta_0\| \geq 2^M} \leq \pr{\| \wh\theta_{n,k}^{(2)} - \theta_0\| \geq \eta_3} + \pr{\sqrt{N}\| \wh\theta_{n,k}^{(1)} - \theta_0\| \geq 2^{M/2}}
\\
+ \pr{\bigcup_{j: j\geq M+1, \ \frac{2^j}{\sqrt N}\leq \eta_3}\sup_{\theta\in S_{N,j},\theta'\in \bar S_{N,M/2}} \sqrt{N}\l| \wh L(\theta',\theta) - L(\theta',\theta)\r| \geq c_1 \frac{2^{2j-2}}{\sqrt N}}
\\
+\pr{ \sup_{\theta\in \bar S_{N,M/2}} \sqrt{N}\l| \wh L(\theta,\theta_0) - L(\theta,\theta_0)\r| \geq c_1 \frac{2^{2M}}{\sqrt N} }.
}
We have shown before that the first and second term on the right side of the previous display converge to $0$ as $M$, $n$ and $k$ tend to infinity, while the last term converges to $0$ in view of argument presented previously in detail (see representation \eqref{eq:b-k-2} and the bounds that follow). 
It remains to estimate $\pr{\bigcup_{j: j\geq M+1, \ \frac{2^j}{\sqrt N}\leq \eta_3}\sup_{\theta\in S_{N,j},\theta'\in \bar S_{N,M/2}} \sqrt{N}\l| \wh L(\theta',\theta) - L(\theta',\theta)\r| \geq c_1 \frac{2^{2j-2}}{\sqrt N}}$. 
To this end, we again invoke Proposition \ref{lemma:b-k} applied to the class $\l\{ \ell(\theta_1,\cdot)-\ell(\theta_2,\cdot), \ \theta_1\in \bar S_{N,M/2}, \theta_2\in \bar S_{N,j}\r\}$; here, the ``reference point'' is $(\theta_0,\theta_0)$. 
Since $|\ell(\theta,x) - \ell(\theta',x)| \leq V(x; r(\theta_0))\frac{2^j+ 2^{M/2}}{\sqrt N}$, it is easy to see that $\sigma^2(\delta)\leq \mb EM^2_{\theta_0}(X) \frac{2^{2j+1} + 2^M}{N} \leq C(\theta_0) \frac{2^{2j}}{N}$, and to deduce that 
\ml{
\sqrt{N}\l( \wh L(\theta',\theta) - L(\theta',\theta) \r)
 \\
 = \frac{\Delta_n}{\mb E \rho'' \l( \frac{\sqrt{n}}{\Delta_n} \l(\bar L_1(\theta') - \bar L_1(\theta) - L(\theta',\theta) \r) \r)}
\frac{1}{\sqrt{k}}\sum_{i=1}^k \rho' \l( \frac{\sqrt{n}}{\Delta_n}\l(\bar L_i(\theta') - \bar L_i(\theta) - L(\theta',\theta) \r)\r)  
\\
+\m R_{n,k,j}(\theta',\theta),
}
where 
\[
\sup_{ \theta\in \bar S_{N,j},\theta'\in \bar S_{N,M/2} }\l| \m R_{n,k,j}(\theta',\theta) \r| \leq C(d,\theta_0) \l(\frac{2^{2j}}{N} \frac{j^4}{\sqrt k} + \frac{2^{3j} j^2}{N^{3/2}} + \sqrt{k}\frac{2^{6j}}{N^3}\r)
\] 
uniformly over all $M\leq j\leq J_{\max}$ with probability at least $1 - \frac{C}{M}$. The remaining steps again closely mimic the argument outlined in detail after display \eqref{eq:b-k-2} and yield that 
\[
\pr{\bigcup_{j: j\geq M+1, \ \frac{2^j}{\sqrt N}\leq \eta_3}\sup_{\theta\in S_{N,j},\theta'\in \bar S_{N,M/2}} \sqrt{N}\l| \wh L(\theta',\theta) - L(\theta',\theta)\r| \geq c_1 \frac{2^{2j-2}}{\sqrt N}}
\leq \frac{C(d,\theta_0)}{\Delta_n}2^{-M+1} \to 0
\]
as $M\to \infty$, therefore implying the last claim in the first part of the proof.

Now we are ready to establish the asymptotic normality of $\wh\theta_{n,k}^{(1)}$ and $\wh\theta_{n,k}^{(2)}$. 
To this end, consider the stochastic process $M_N(h,q)$ indexed by $h,q\in \mb R^d$ and defined via
\[
M_N(h,q):=N\l( \wh L(\theta_0 + h/\sqrt N, \theta_0 + q/\sqrt N)) - L(\theta_0 + h/\sqrt N,\theta_0 + q/\sqrt N) \r).
\]
Below, we will show that $M_N(h,q)$ converges weakly to the Gaussian process $W(h,q):=W^T(h-q)$, $h,q\in \mb R^d$, where $W\sim N(0,\Sigma_W)$ and $\Sigma_W = \mb E \l[ \pd_\theta \ell(\theta_0,X) \pd_\theta \ell(\theta_0,X)^T\r]$. 
Let us deduce the conclusion assuming that weak convergence has already been established. 
We have that 
\be
N \cdot\wh L(\theta_0 + h/\sqrt N, \theta_0 + q/\sqrt N)) = N\cdot L(\theta_0 + h/\sqrt N,\theta_0 + q/\sqrt N) + M_N(h,q).
\ee
Note that, in view of assumption \ref{ass:2} and the fact that $\theta_0$ minimizes $L(\theta_0)$, 
\[
N\cdot L(\theta_0 + h/\sqrt N,\theta_0 + q/\sqrt N) \to \frac{1}{2} h^T \pd^2_\theta L(\theta_0) h - \frac{1}{2} q^T \pd^2_\theta L(\theta_0) q \text{ as } N\to\infty,
\]
therefore $N \cdot \wh L(\theta_0 + h/\sqrt N, \theta_0 + q/\sqrt N)) \xrightarrow{d} W^T h +  \frac{1}{2} h^T \pd^2_\theta L(\theta_0) h - \l( W^T q - \frac{1}{2} q^T \pd^2_\theta L(\theta_0) q\r)$. 
It is easy to see that 
\ml{
\l( - \l[\pd^2_\theta L(\theta_0)\r]^{-1} W, - \l[\pd^2_\theta L(\theta_0)\r]^{-1} W\r) 
\\
= \argmin_h\max_{q}W^T h +  \frac{1}{2} h^T \pd^d_\theta L(\theta_0) h - \l( W^T q - \frac{1}{2} q^T \pd^d_\theta L(\theta_0) q\r),
}
where $-\l[\pd^2_\theta L(\theta_0)\r]^{-1} W \sim N\l(0,\l[\pd^2_\theta L(\theta_0)\r]^{-1}\Sigma_W\l[\pd^2_\theta L(\theta_0)\r]^{-1}\r)$. 
Therefore, since 
\[
\l(\sqrt{N}\l(\wh \theta_{n,k}^{(1)} - \theta_0\r), \sqrt{N}\l(\wh \theta_{n,k}^{(2)} - \theta_0\r)\r)=\argmin_h \max_q \wh L(\theta_0 + h/\sqrt N, \theta_0 + q/\sqrt N)),
\] 
continuous mapping theorem yields the desired conclusion.

\noindent Next, we will establish the required weak convergence. 
To this end, we apply Proposition \ref{lemma:b-k} to the class 
\ben
\label{eq:new-class}
\wt{\m L}_N:=\l\{ \wt\ell_N(h,q,\cdot):=\ell(\theta_0 + h/\sqrt N,\cdot) - \ell(\theta_0 + q/\sqrt N, \cdot), \ \l\|  \begin{pmatrix} h \\ q \end{pmatrix}\r\|\leq R \r\},
\een
and note that $\l\| \begin{pmatrix} \theta_0 + h/\sqrt N \\ \theta_0 + q/\sqrt N \end{pmatrix} -  \begin{pmatrix} \theta_0  \\ \theta_0 \end{pmatrix} \r\| \leq \frac{R}{\sqrt N}$. 
We will also introduce the following notation for brevity (that will be used only in this part of the proof): 
\begin{align}
\nonumber
\bar L_j(h,q)&:=\frac{1}{n}\sum_{i\in G_j} \wt\ell_N(h,q,X_i), 
\\
\label{eq:notation}
\wt L(h,q)&:=\mb E\wt\ell_N(h,q,X).
\end{align}
The quantities $\delta$ and $\sigma^2(\delta)$ defined in Proposition \ref{lemma:b-k} admit the bounds 
$\delta\leq \frac{R}{\sqrt N}$ and, in view of assumption \ref{ass:3},
\ben
\label{eq:var-delta}
\sigma^2(\delta):=\sup_{\l\|  (h, q)^T \r\|\leq R} \var\l( \wt\ell_N(h,q,X)\r) 
\leq 2\mb E V^2(X; r(\theta_0)) \frac{R^2}{N},
\een
hence Proposition \ref{lemma:b-k} yields that 
\ml{
M_N(h,q) = \frac{\Delta_n}{\mb E \rho'' \l(  \frac{\sqrt{n}}{\Delta_n}\l(\bar L_1(h,q) - \wt L(h,q) \r)\r)}
\frac{\sqrt{N}}{\sqrt{k}}\sum_{j=1}^k \rho' \l( \frac{\sqrt{n}}{\Delta_n}\l(\bar L_j(h,q) - \wt L(h,q) \r) \r) + o_P(1)
}
uniformly over $\l\|  (h, q)^T \r\|\leq R$. In view of assumption \ref{ass:1}, 
\[
\pr{\l|  \frac{\sqrt{n}}{\Delta_n}\l(\bar L_j(h,q) - \wt L(h,q) \r) \r| \leq 1}\leq \mb E \rho'' \l(  \frac{\sqrt{n}}{\Delta_n}\l(\bar L_1(h,q) - \wt L(h,q) \r)\r) \leq 1.
\]
As $\sup_{\l\|  (h, q)^T \r\|\leq R} \pr{\l|  \frac{\sqrt{n}}{\Delta_n}\l(\bar L_j(h,q) - \wt L(h,q) \r) \r| \geq 1}\leq \sup_{\l\|  (h, q)^T \r\|\leq R} \frac{\var\l( \wt \ell(h,q,X\r)}{\Delta_n^2}\to 0$ as $n,k\to \infty$, we deduce that 
$\mb E \rho'' \l(  \frac{\sqrt{n}}{\Delta_n}\l(\bar L_1(h,q) - \wt L(h,q) \r)\r)\to 1$ and 
\ben
\label{eq:lin}
M_N(h,q) = \Delta_n
\frac{\sqrt{N}}{\sqrt{k}}\sum_{j=1}^k \rho' \l( \frac{\sqrt{n}}{\Delta_n}\l(\bar L_j(h,q) - \wt L(h,q) \r) \r) + o_P(1).
\een	
It remains to establish convergence of the finite dimensional distributions as well as asymptotic equicontinuity. 
Convergence of finite dimensional distributions will be deduced from Lindeberg-Feller's central limit theorem. 
As $\rho'(x)=x$ for $|x|\leq 1$ by assumption \ref{ass:1}, 
\[
\rho' \l( \frac{\sqrt{n}}{\Delta_n}\l(\bar L_j(h,q) - \wt L(h,q) \r) \r) = \frac{\sqrt{n}}{\Delta_n}\l(\bar L_j(h,q) - \wt L(h,q) \r)
\]
on the event $\m C_j:=\l\{ \l|  \frac{\sqrt{n}}{\Delta_n}\l(\bar L_j(h,q) - \wt L(h,q) \r) \r| \leq 1\r\}$. 
Chebyshev's inequality and assumption \ref{ass:3} imply that 
\be
\pr{\bar{\m C}_j} \leq \var \l( \frac{\sqrt{n}}{\Delta_n}\l(\bar L_j(h,q) - \wt L(h,q) \r) \r) \leq \frac{\mb E \wt \ell^2(h,q,X)}{\Delta_n^2}
\leq \frac{\mb E \m V^2(X; r(\theta_0)) \|h-q\|^2}{\Delta_n^2 N},
\ee
therefore, $\pr{\bigcup_{j=1}^k \bar{\m C}_j }\leq \frac{\mb E \m V^2(X; r(\theta_0)) \|h-q\|^2}{\Delta_n^2 n}\to 0$ as $n\to\infty$, and 
\ml{
M_N(h,q) = \Delta_n \frac{\sqrt N}{\sqrt k}\sum_{j=1}^k \frac{\sqrt{n}}{\Delta_n}\l(\bar L_j(h,q) - \wt L(h,q) \r) +o_P(1)
\\
= \frac{1}{\sqrt N}\sum_{j=1}^N \sqrt{N}\l( \wt \ell_N(h,q,X_j) - \wt L(h,q)\r) + o_P(1)
}
on the event $\bigcap_{j=1}^k \m C_j$. 
Hence, the limits of the finite dimensional distributions of the processes 
$M_N(h,q)$ and $\wh M_N(h,q):=\frac{1}{\sqrt N}\sum_{j=1}^N \sqrt{N}\l( \wt \ell_N(h,q,X_j) - \wt L(h,q)\r)$ coincide. 
It is easy to conclude from the Lindeberg-Feller's theorem that the finite dimensional distributions of the process $(h,q)\mapsto\wh M_N(h,q)$ are Gaussian, with covariance function
\ben
\label{eq:covariance}
\lim_{N\to\infty}\cov\l( \wh M_N(h_1,q_1), \wh M_N(h_2,q_2)\r) = \l( h_1 - q_1\r)^T \mb E\l[ \pd_\theta \ell(\theta_0,X) \l( \pd_\theta \ell(\theta_0,X)\r)^T\r]  \l( h_2 - q_2\r),
\een
Indeed, the aforementioned relation follows from the dominated convergence theorem, where pointwise convergence and the ``domination''  hold due to assumption \ref{ass:3}. Lindeberg's condition is also easily verified, as 
$\l(\sqrt N\wt\ell_N(h,q,X)\r)^2 \leq \m V^2(X; r(\theta_0))\|h-q\|^2$, implying that the sequence $\l\{ \l(\sqrt{N_j}\,\wt\ell_{N_j}(h,q,X)\r)^2 \r\}_{j\geq 1}$ is uniformly integrable, where $N_j = n_j\cdot k_j$.

Finally, we will establish the asymptotic equicontinuity of the process $M_N(h,q)$. To this end, it suffices to prove that for any $\eps>0$, 
\[
\lim_{\delta\to 0}\limsup_{n,k\to\infty} \pr{ \sup_{\|(h_1,q_1)^T - (h_2,q_2)^T\|\leq \delta}\l| M_N(h_1,q_1) - M_N(h_2,q_2)\r| \geq \eps}\to 0,
\]
which would follow, in view of Proposition \ref{lemma:b-k}, from the relation
\mln{
\label{eq:e-cont}
\lim_{\delta\to 0}\limsup_{n,k\to\infty}\mb E\sup_{\|(h_1,q_1)^T - (h_2,q_2)^T\|\leq \delta}
\Bigg|\Delta_n
\frac{\sqrt{N}}{\sqrt{k}}\sum_{j=1}^k \Bigg( \rho' \l( \frac{\sqrt{n}}{\Delta_n}\l(\bar L_j(h_1,q_1) - \wt L(h_1,q_1) \r) \r)
\\
-  \rho' \l( \frac{\sqrt{n}}{\Delta_n}\l(\bar L_j(h_2,q_2) - \wt L(h_2,q_2) \r) \r)\Bigg) \Bigg| = 0.
}
To estimate the expected supremum in \eqref{eq:e-cont}, we first observe that for any $h,q$, 
\ben
\label{eq:unbias}
\sqrt{Nk}\l| \mb E \rho' \l( \frac{\sqrt{n}}{\Delta_n}\l(\bar L_1(h,q) - \wt L(h,q) \r) \r)\r| = o(1)
\een 
as $k,n\to\infty$ by Lemma \ref{lemma:lindeberg} and inequality \eqref{eq:var-delta}. Therefore, we only need to show that
\ml{
\limsup_{n,k\to\infty}\mb E\sup_{\|(h_1,q_1)^T - (h_2,q_2)^T\|\leq \delta}\l| M_N(h_1,q_1) - M_N(h_2,q_2) - \r.
\\
\l.\l(\mb EM_N(h_1,q_1) - \mb EM_N(h_2,q_2)\r)\r| \xrightarrow{\delta\to 0}0.
}
Next, we will apply symmetrization inequality with Gaussian weights \citep{wellner1}. 
Specifically, let $g_1,\ldots,g_k$ be i.i.d. $N(0,1)$ random variables independent of the data $X_1,\ldots,X_N$. 
Then, setting $B(\delta):=\l\{ (h_1,q_1),(h_2,q_2): \ \|(h_1,q_1)^T - (h_2,q_2)^T\|\leq \delta\r\}$, we have that 
\ml{
\mb E\sup_{B(\delta)}\l| M_N(h_1,q_1) - M_N(h_2,q_2) - \l(\mb EM_N(h_1,q_1) - \mb EM_N(h_2,q_2)\r)\r| \leq 
\\
C(\rho)\Delta_n \mb E\sup_{B(\delta)}
\Bigg|
\frac{\sqrt{N}}{\sqrt{k}}\sum_{j=1}^k g_j\Bigg( \rho' \l( \frac{\sqrt{n}}{\Delta_n}\l(\bar L_j(h_1,q_1) - \wt L(h_1,q_1) \r) \r) 
\\
-  \rho' \l( \frac{\sqrt{n}}{\Delta_n}\l(\bar L_j(h_2,q_2) - \wt L(h_2,q_2) \r) \r) \Bigg) \Bigg|.
}
Let us condition everything on $X_1,\ldots,X_N$; we will write $\mb E_g$ to denote the expectation with respect to $g_1,\ldots,g_k$ only. Consider the Gaussian process $Y_{n,k}(t)$ defined via 
\[
\mb R^k \ni t \mapsto Y_{n,k}(t):=\frac{1}{\sqrt{k}}\sum_{j=1}^k g_j \sqrt{N}\rho'(t_{j}),
\]
where $t_j := t_j(h,q) = \frac{\sqrt{n}}{\Delta_n}\l(\bar L_j(h,q) - \wt L(h,q) \r),  \ j=1,\ldots,k$. In what follows, we will rely on the ideas behind the proof of Theorem 2.10.6 in \cite{wellner1}. 
Let us partition the set $\{ (h,q): \|(h,q)\|\leq R\}$ into the subsets $S_j, \ j=1,\ldots,N(\delta)$ of diameter at most 
$\delta$ with respect to the Euclidean distance $\|\cdot\|$, and let $t^{(j)}:=t^{(j)}(h^{(j)},q^{(j)})\in S_j \ j=1,\ldots,N(\delta)$ be arbitrary points; we also note that  $N(\delta) \leq \l( \frac{6R}{\delta}\r)^{2d}$. 
Next, set $T^{(j)}:=\{ t(h,q): \ (h,q)\in S_j \}$.   
Our goal will be to show that 
\[
\limsup_{n,k\to\infty} \mb E\max_{j=1,\ldots,N(\delta)}\sup_{t\in T^{(j)}}\l| Y_{n,k}(t) - Y_{n,k}(t^{(j)}) \r| \to 0 \text{ as } \delta\to 0,
\]
whence the desired conclusion would follow from Theorem 1.5.6 in \cite{wellner1}. 
By Lemma 2.10.16 in \cite{wellner1}, 
\mln{
\label{eq:partition}
\mb E_g\max_{j=1,\ldots,N(\delta)}\sup_{t\in T^{(j)}}\l| Y_{n,k}(t) - Y_{n,k}(t^{(j)}) \r| 
\leq C\Bigg( \max_{j=1,\ldots,N(\delta)} \mb E_g\sup_{t\in T^{(j)}}\l| Y_{n,k}(t) - Y_{n,k}(t^{(j)}) \r|
\\
+ \log^{1/2}N(\delta) \max_{1\leq j\leq N(\delta)} \sup_{t\in T^{(j)}} \var^{1/2}_g\l( Y_{n,k}(t) - Y_{n,k}(t^{(j)})\r)\Bigg).
}
Observe that $\var_g\l( Y_{n,k}(t) - Y_{n,k}(t^{(j)})\r) = \frac{N}{k}\sum_{i=1}^k \l( \rho'(t_i) - \rho'(t^{(j)}_i))\r)^2$, therefore, 
\ml{
\mb E \max_{1\leq j\leq N(\delta)} \sup_{t\in T^{(j)}} \var_g^{1/2}\l( Y_{n,k}(t) - Y_{n,k}(t^{(j)})\r) 
\leq \mb E^{1/2} \sup_{t^{(1)},t^{(2)} }  \frac{N}{k}\sum_{i=1}^k \l( \rho'(t^{(1)}_i) - \rho'(t^{(2)}_i) \r)^2
\\
\leq \sqrt{N} L(\rho') \mb E^{1/2} \sup_{t^{(1)},t^{(2)} } \l(t^{(1)}_1 -t^{(2)}_1 \r)^2 
\\
= L(\rho') \mb E^{1/2} \sup_{\|(h_1,q_1) - (h_2,q_2)\|\leq \delta} \l( \frac{\sqrt{nN}}{\Delta_n}\l(\bar L_1(h_1,q_1) - \bar L_1(h_2,q_2) - (\wt L(h_1,q_1) - \wt L(h_2,q_2) \r) \r)^2,
}
where the supremum is taken over all $t^{(1)}(h_1,q_1), \ t^{(2)}(h_2,q_2)$ such that $\|(h_1,q_1) - (h_2,q_2)\|\leq \delta$. 
To estimate the last expected supremum, we invoke Lemma \ref{lemma:sup-power} with $f_{h,q}(X):=\ell(\theta_0+h/\sqrt N,X) - \ell(\theta_0+q/\sqrt N,X)$, noting that, in view of assumption \ref{ass:3}, 
\mln{
\label{eq:final}
\sqrt{N} |f_{h_1,q_1}(X) - f_{h_2,q_2}(X)|\leq \m V(X; r(\theta_0))\l( \|h_1-h_2\| + \|q_1 - q_2\|\r)
\\
\leq 2 \m V(X; r(\theta_0))\l\| (h_1,q_1) - (h_2,q_2)\r\|.
}
Therefore, 
\ml{
\mb E^{1/2} \sup_{\|(h_1,q_1) - (h_2,q_2)\|\leq \delta} \l( \frac{\sqrt{nN}}{\Delta_n}\l(\bar L_1(h_1,q_1) - \bar L_1(h_2,q_2) - (\wt L(h_1,q_1) - \wt L(h_2,q_2) \r) \r)^2 
\\
\leq C\sqrt{d}\mb E^{1/2} \m V^2(X; r(\theta_0))\cdot \delta,
}
yielding that the second term on the right side of \eqref{eq:partition} converges in probability to $0$ as $\delta\to 0$. 
It remains to show that the first term $\max_{j=1,\ldots,N(\delta)} \mb E_g\sup_{t\in T^{(j)}}\l| Y_{n,k}(t) - Y_{n,k}(t^{(j)}) \r|$ converges to $0$ in probability.
As $\rho'$ is Lipschitz continuous, the covariance function of $Y_{n,k}(t)$ satisfies 
\[
\mb E\l( Y_{n,k}(t^{(1)}) - Y_{n,k}(t^{(2)})\r)^2 \leq L^2(\rho')\frac{N}{k}\sum_{j=1}^k \l( t^{(1)}_{j} - t^{(2)}_j \r)^2, 
\]
where the right side corresponds to the variance of increments of the process 
\[
Z_{n,k}(t)=\frac{L(\rho')}{\sqrt k}\sum_{j=1}^k g_j \sqrt{N}t_j.
\] 
Therefore, Slepian's lemma \citep{LT-book-1991} implies that for any $j$,
\ml{
\mb E_g\sup_{t\in T^{(j)}}\l| Y_{n,k}(t) - Y_{n,k}(t^{(j)}) \r| 
\\
\leq \mb E_g \sup_{(h,q)\in S_j} \frac{1}{\sqrt k}\l| \frac{\sqrt{Nn}}{\Delta_n}\sum_{i=1}^k g_j\l( \bar L_i(h,q) -\bar L_i(h^{(j)},q^{(j)}) - (\wt L(h,q) - L(h^{(j)},q^{(j)}))\r) \r|.
}
In turn, it yields the inequality 
\ml{
\mb E\max_{j=1,\ldots,N(\delta)} \mb E_g\sup_{t\in T^{(j)}}\l| Y_{n,k}(t) - Y_{n,k}(t^{(j)}) \r|
\\
\leq \mb E \sup_{\|(h_1,q_1) - (h_2,q_2)\|\leq \delta}  \frac{1}{\sqrt k}\l| \frac{\sqrt{Nn}}{\Delta_n}\sum_{i=1}^k g_j\l( \bar L_i(h_1,q_1) -\bar L_i(h_2,q_2) - (\wt L(h_1,q_1) - \wt L(h_2,q_2))\r) \r|.
}
To complete the proof, we will apply the multiplier inequality \citep[Lemma 2.9.1 in][]{wellner1} to deduce that the last display is bounded, up to a multiplicative constant, by
\[
\max_{m=1,\ldots,k} \mb E \sup_{\|(h_1,q_1) - (h_2,q_2)\|\leq \delta}  \frac{1}{\sqrt m}\l| \frac{\sqrt{Nn}}{\Delta_n}\sum_{i=1}^m \eps_j\l( \bar L_i(h_1,q_1) -\bar L_i(h_2,q_2) - (\wt L(h_1,q_1) - \wt L(h_2,q_2))\r) \r|
\]
where $\eps_1,\ldots,\eps_k$ are i.i.d. Rademacher random variables. Next, desymmetrization inequality \citep[Lemma 2.3.6 in][]{wellner1} implies that for any $m=1,\ldots,k$, 
\ml{
\mb E \sup_{\|(h_1,q_1) - (h_2,q_2)\|\leq \delta}  \frac{1}{\sqrt m}\l| \frac{\sqrt{Nn}}{\Delta_n}\sum_{i=1}^m \eps_j\l( \bar L_i(h_1,q_1) -\bar L_i(h_2,q_2) - (\wt L(h_1,q_1) - \wt L(h_2,q_2))\r) \r|
\\
\leq 2\mb E \sup_{\|(h_1,q_1) - (h_2,q_2)\|\leq \delta}  \frac{1}{\sqrt mn }\l| \frac{\sqrt{N}}{\Delta_n}\sum_{i=1}^{mn} 
\l( \wt\ell_N(h_1,q_1,X_i) -  \wt\ell_N(h_2,q_2,X_i) - (\wt L(h_1,q_1) - \wt L(h_2,q_2))\r) \r|
}
where $\wt\ell_N(h,q,X)$ and $\wt L(h,q)$ were defined in \eqref{eq:new-class} and \eqref{eq:notation} respectively. 
It remains to apply Lemma \ref{lemma:sup-power} in exactly the same way as before (see \eqref{eq:final}) to deduce that the last display is bounded from above by $C \sqrt{d} \mb E^{1/2} \m V^2(X; r(\theta_0))\cdot \delta\to 0$ as $\delta\to 0$. 
This completes the proof of asymptotic equicontinuity, and therefore weak convergence, of the sequence of processes $M_N(h,q)$.


\subsection{Proof of Lemma \ref{lemma:unif}.}
\label{sec:proof-unif}

Define
\[
G_k(z;\theta) = \frac{1}{\sqrt{k}}\sum_{j=1}^k \rho'\l(\sqrt{n}\,\frac{\bL_j (\theta) - \bL_j(\theta_0) - L(\theta,\theta_0) - z}{\Delta_n}\r),
\]
and recall that the contaminated sample $X_1,\ldots,X_N$ contains $\m O$ outliers; let $I\subset \{1,\ldots,N\}$ denote the index set of the outliers. Moreover, let $\tilde X_1,\ldots,\tilde X_N$ be an i.i.d. sample from $P$ such that $\tilde X_j \equiv X_j$ for $j\notin I$, and let $\tilde G_k(z;\theta)$ be a version of $G_k(z;\theta)$ based on the uncontaminated sample. Clearly, 
$\l| G_k(z;\theta) - \tilde G_k(z;\theta)\r|\leq 2\|\rho\|_\infty\frac{\m O}{\sqrt k}$ almost surely, for all $z\in \mb R$. 
  
Suppose that $z_1,z_2\in \mb R$ are such that on an event of probability close to $1$, $G_k(z_1;\theta) > 0$ and $G_k(z_2;\theta) < 0$ for all $\theta\in \Theta$ simultaneously. 
Since $G_k$ is non-increasing in $z$, it is easy to see that on this event, $\hL(\theta,\theta_0)\in (z_1,z_2)$ for all $\theta\in \Theta$, implying that 
\begin{equation}
\label{eq:a11}
\sup_{\theta\in \Theta'}\l| \hL(\theta,\theta_0) - L(\theta,\theta_0) \r|\leq \max(|z_1|,|z_2|).
\end{equation} 
Our goal is to find $z_1,z_2$ satisfying conditions above and such that $|z_1|,\, |z_2|$ are as small as possible. 
Let $W(\theta)$ stand for a centered normally distributed random variable with variance $\sigma^2(\theta,\theta_0)$, and observe that 
\[
G_k(z;\theta) = A_0+A_1 + A_2 + A_3,
\]
where 
\begin{align*}
A_0(\theta) &= G_k(z;\theta) - \tilde G_k(z;\theta), 
\\
A_1(\theta) &= \frac{1}{\sqrt{k}}\sum_{j=1}^k \l( \rho'\l(\sqrt{n}\,\frac{\bL_j (\theta) - \bL_j(\theta_0) - L(\theta,\theta_0) - z}{\Delta_n}\r) - \mb E \rho'\l(\sqrt{n}\,\frac{\bL_j (\theta) - \bL_j(\theta_0) - L(\theta,\theta_0) - z}{\Delta_n}\r) \r),
\\
A_2(\theta) & = \sqrt{k}\l( \mb E \rho'\l(\sqrt{n}\,\frac{\bL_1 (\theta) - \bL_1(\theta_0) - L(\theta,\theta_0) - z}{\Delta_n}\r) - \mb E \rho'\l( \frac{W(\theta) - \sqrt{n} z }{\Delta_n} \r)\r), 
\\
A_3(\theta) & = \sqrt{k} \mb E\rho'\l( \frac{W(\theta) - \sqrt{n} z }{\Delta_n} \r).
\end{align*}
With some abuse of notation, we assume that $A_1(\theta)$ is based on the sample $\tilde X_1,\ldots,\tilde X_N$. Next, suppose that $\eps_1,\eps_2$ are positive and such that $\inf_{\theta\in \Theta'} A_0(\theta)> -\eps_0$, $\inf_{\theta\in \Theta'} A_1(\theta)> -\eps_1$ with high probability and $\inf_{\theta\in \Theta'} A_2(\theta)> -\eps_2$. Then $z_1$ satisfying $\inf_{\theta\in \Theta'} \mb E\rho'\l( \frac{W(\theta) - \sqrt{n} z_1 }{\Delta_n} \r)\geq\frac{\eps_0+\eps_1+\eps_2}{\sqrt k}$ will conform to our requirements. Since $\mb E\rho'\l( \frac{W(\theta) - \sqrt{n} z_1 }{\Delta_n} \r) \approx \underbrace{\mb E\rho'\l( \frac{W(\theta) }{\Delta_n} \r)}_{=0} - \mb E\rho''\l( \frac{W(\theta)}{\Delta_n} \r)\frac{\sqrt{n}z_1}{\Delta_n}$ for small $z_1$, a natural choice is $z_1\approx \frac{\Delta_n}{\inf_{\theta\in\Theta'}\mb E\rho''\l( \frac{W(\theta)}{\Delta_n} \r)} \frac{\eps_0+\eps_1+\eps_2}{\sqrt{nk}}$. 
This argument is made precise in \citep[][Lemma 4.3]{minsker2018uniform} which shows that the choice
\[
z_1 = -\frac{\eps_0+\eps_1+\eps_2}{ 0.09}\frac{\widetilde \Delta}{\sqrt{nk}}
\]
is sufficient whenever $\eps_j,\ j=0,1,2$ are not too large (specifically, when $\frac{\eps_0+\eps_1+\eps_2}{\sqrt k}\leq 0.045$ - this is precisely the main condition needed for the bound of lemma to hold). 
It remains to provide the values for $\eps_j, \ j=0,1,2$. We have already shown above that $\eps_0$ can be chosen as $\eps_0=2\|\rho\|_\infty\frac{\m O}{\sqrt k}$. To find a feasible value of $\eps_1$, we will apply Markov's inequality stating that with probability at least $1-1/s$,
\[
\sup_{\theta\in\Theta'}|A_1(\theta)| \leq s \,\mb E \sup_{\theta \in \Theta'} \l| \frac{1}{\sqrt k}\sum_{j=1}^k \rho'\l( \frac{\sqrt{n}}{\Delta_n}\l(\bar L_j(\theta) - L(\theta)\r) \r) - \mb E \rho' \l( \frac{\sqrt{n}}{\Delta_n}\l(\bar L_1(\theta) - L(\theta)\r) \r) \r|.
\]
The expected supremum can be estimated in a standard way using the symmetrization, contraction and desymmetrization inequalities (e.g. see the proof of proposition \ref{lemma:denom} below), yielding that 
\begin{multline*}
\mb E \sup_{\theta \in \Theta'} \l| \frac{1}{\sqrt k}\sum_{j=1}^k \rho'\l( \frac{\sqrt{n}}{\Delta_n}\l(\bar L_j(\theta) - L(\theta)\r) \r) - \mb E \rho' \l( \frac{\sqrt{n}}{\Delta_n}\l(\bar L_1(\theta) - L(\theta)\r) \r) \r|
\\
\leq \frac{8L(\rho')}{\Delta_n} \mb E\sup_{\theta\in \Theta'} \frac{1}{\sqrt{N}}\l| \sum_{j=1}^N \l(\ell(\theta,X_j) - \ell(\theta_0,X_j) - L(\theta,\theta_0)\r) \r|. 
\end{multline*}
It remains to obtain an appropriate value for $\eps_2$. Note that for any bounded non-negative function $g:\mb R\mapsto \mb R_+$ and any signed measure $Q$, 
\begin{align*}
\l| \int_{\mb R} g(x)dQ \r| = \l| \int_{0}^{\|f\|_\infty} Q\l( x: \,g(x)\geq t \r) dt \r|
\leq \|g\|_\infty \max_{t\geq 0} \l|  Q\l( x: \,g(x)\geq t \r)\r|.
\end{align*} 
Moreover, if $g$ is monotone, the sets $\{x: \,g(x)\geq t\}$ and $\{x: \,g(x)\leq t\}$ are half-intervals. 
Note that $\rho'= \max(\rho',0) - \max(-\rho', 0)$ is a difference of two non-negative monotone functions. Therefore,
\[
\l| \int_{\mb R} \rho'\l(\frac{x-\sqrt{n}z}{\Delta_n}\r)dQ(x) \r|\leq \|\rho'\|_\infty \l( \max_{t\geq 0} \l|  Q\l( x: \,\rho'(x)\geq t \r)\r| + \max_{t\leq 0} \l|  Q\l( x: \,\rho'(x)\leq t \r)\r|\r).
\] 
Take $Q$ to be the difference of the distributions of $\sqrt{n}\,\l(\bL_1 (\theta) - \bL_1(\theta_0) - L(\theta,\theta_0)\r)$ and $W(\theta)$, denoted $\Phi_\theta^{(n,k)}$ and $\Phi_\theta$ respectively, so that 
\begin{multline*}
\sqrt{k}\l( \mb E \rho'\l(\sqrt{n}\,\frac{\bL_1 (\theta) - \bL_1(\theta_0) - L(\theta,\theta_0) - z}{\Delta_n}\r) - \mb E \rho'\l( \frac{W(\theta) - \sqrt{n} z }{\Delta_n} \r)\r) 
\\
\leq 
2\sqrt k \|\rho'\|_\infty  \sup_{t\in \mb R} \l| \Phi^{(n,k)}_\theta(t) - \Phi_\theta(t) \r|. 
\end{multline*}
A well-known result by \citet{feller1968berry} states that $\sup_{t\in \mb R} \l| \Phi^{(n,k)}_\theta(t) - \Phi_\theta(t) \r| \leq 6g_\theta(n)$, where 
\[
g_\theta(n):= \frac{1}{\sqrt n} \mb E\l[ \l( \frac{\ell(\theta,X) - \ell(\theta_0,X) - L(\theta,\theta_0)}{\sigma(\theta,\theta_0)}\r)^2 \min\l(\l| \frac{\ell(\theta,X) - \ell(\theta_0,X) - L(\theta,\theta_0)}{\sigma(\theta,\theta_0)} \r|, \sqrt n\r) \r],
\]
It is easy to see that $g_\theta(n)\to 0$ as $n\to\infty$ if $\var(\ell(\theta,X))<\infty$, and  distributions with finite variance, and moreover $g_\theta(n)\leq C \mb E\l|  \frac{\ell(\theta,X) - \ell(\theta_0,X) - L(\theta,\theta_0)}{\sigma(\theta,\theta_0)} \r|^\tau n^{-\tau/2}$ if $\mb E\l|  \frac{\ell(\theta,X) - \ell(\theta_0,X) - L(\theta,\theta_0)}{\sigma(\theta,\theta_0)} \r|^{2+\tau}<\infty$ for some $\tau\in(0,1]$. 
Therefore, the function $g_{\tau}(n,\theta)$ in the statement of the lemma can be chosen as $g_{\tau}(n,\theta) = g_\theta(n)$ when $\tau=0$ and $g_{\tau}(n,\theta) = C$ when $\tau>0$. 
We conclude that the choice $\eps_2 = 12\sqrt k \|\rho'\|_\infty \sup_{\theta\in\Theta'} g_\theta(n)$ satisfies the desired requirements. 

It remains to recall the bound \eqref{eq:a11} and that $z_1 = -\frac{\eps_0+\eps_1+\eps_2}{ 0.09}\frac{\widetilde \Delta}{\sqrt{nk}}$. The matching bound for $z_2$ is obtained in an identical fashion.

\begin{remark}
The bound for $\eps_2$ that we established above is slightly weaker than the one used in the statement of the lemma; an improved version can be obtained using the non-uniform version of the Berry-Esseen bound with additional effort, and we refer the reader to \citep[][Lemma 4.2]{minsker2018uniform} for the technical details.
\end{remark}

The following proposition is one of the key technical results that the proof of Theorem \ref{th:normality-B} relies on.
\begin{proposition}
\label{lemma:b-k}
Let $\m L=\{ \ell(\theta,\cdot), \ \theta\in \Theta\}$ be a class of functions, and, given $\theta_0\in \Theta$, set 
$\sigma^2(\delta):=\sup_{\|\theta - \theta_0\|\leq \delta} \var\l( \ell(\theta,X)-\ell(\theta_0,X)\r)$. Moreover, let assumption \ref{ass:3} hold. Then for every $\delta \leq r(\theta_0)$ \footnote{$r(\theta_0)$ was defined in the paragraph following assumption \ref{ass:3}}, the following representation holds uniformly over $\|\theta-\theta_0\|\leq \delta $: 
\mln{
\label{eq:b-k}
\sqrt{N}\l(\hL(\theta,\theta_0) - L(\theta,\theta_0)\r) 
\\
=
\frac{\Delta_n}{\mb E \rho'' \l( \frac{\sqrt{n}}{\Delta_n}\l(\bar L_1(\theta,\theta_0) - L(\theta,\theta_0)\r) \r)}
\frac{1}{\sqrt{k}}\sum_{j=1}^k \rho' \l( \frac{\sqrt{n}}{\Delta_n}\l(\bar L_j(\theta,\theta_0) - L(\theta,\theta_0) \r) \r) 
+ \m R_{n,k,j}(\theta),
}
where 
\[
\sup_{\|\theta-\theta_0\|\leq \delta }\l| \m R_{n,k,j}(\theta)\r|\leq C(d,\theta_0)\l( \delta^2 \frac{s^2}{\sqrt{k}}  + 
s\delta^3 +  \sqrt{k} \delta^6\r)
\]
with probability at least $1-\frac{3}{s}$.
\end{proposition}
\begin{proof}
First, observe that in view of assumption \ref{ass:3},
\begin{equation*}
\sigma^2(\delta)\leq \sup_{\|\theta - \theta_0\|\leq \delta}\mb E\lvert \ell(\theta,X) - \ell(\theta_0,X)\rvert^2 
\leq \mb E\m V^2(X;r(\theta_0))\, \delta^2.
\end{equation*}
Next, define 
\[
\wh G_k(z;\theta):=\frac{1}{k} \sum_{j=1}^k \rho' \l( \frac{\sqrt{n}}{\Delta_n}\l(\bar L_j(\theta,\theta_0) - L(\theta,\theta_0) - z \r) \r) 
\]
so that $\wh G_k(\hL(\theta,\theta_0) - L(\theta,\theta_0);\theta) = 0$, and let $G_k(z;\theta):=\mb E \wh G_k(z;\theta)$. 
Next, consider the stochastic process
\[
R_k(\theta) = \wh G_k(0;\theta) + \pd_z G_k(z;\theta)\big|_{z=0} \l(\hL(\theta,\theta_0) - L(\theta,\theta_0) \r).
\]
We claim that for any $\theta\in \Theta$, 
\ben
\label{eq:R}
\sqrt{N}\frac{R_k(\theta')}{ \pd_z G_k(z;\theta')|_{z=0}} = O_P\l(  \frac{\delta^2}{\sqrt{k}}  + 
\delta^3
 +  \sqrt{k} \delta^6 \r)
\een
uniformly over $\theta'$ in the neighborhood of $\theta$. 
Taking this claim for granted for now, we see that 
$\sqrt{N}\l(\hL(\theta,\theta_0) - L(\theta,\theta_0)\r) = - \sqrt{N} \frac{\wh G_k(0;\theta)}{ \pd_z G_k(z;\theta)|_{z=0}} + \sqrt{N}\frac{ R_k(\theta)}{ \pd_z G_k(z;\theta)|_{z=0}}$, 
and in particular it follows from the claim above that the weak limits of $\sqrt{N}(\hL(\theta,\theta_0) - L(\theta,\theta_0))$ and 
\[
 - \sqrt{N} \frac{\wh G_k(0;\theta)}{ \pd_z G_k(z;\theta)|_{z=0}} = \frac{\Delta_n}{\sqrt{k}} \, \frac{ \sum_{j=1}^k \rho' \l( \frac{\sqrt{n}}{\Delta_n}\l(\bar L_j(\theta,\theta_0) - L(\theta,\theta_0) \r) \r) }{\mb E \rho'' \l( \frac{\sqrt{n}}{\Delta_n}\l(\bar L_1(\theta,\theta_0) - L(\theta,\theta_0)\r) \r)}.
\]
coincide whenever $\delta$ sufficiently small (note that we can change the order of differentiation and expectation in the denominator as $\rho''$ is bounded). It remains to establish the relation \eqref{eq:R}. To this end, define $\wh e_N(\theta):= \hL(\theta,\theta_0) - L(\theta,\theta_0)$ so that $G_k(\wh e_N(\theta);\theta) = 0$, and let 
$G_k(z;\theta):=\mb E \wh G_k(z)$. 
Recall the definition of $R_k(\theta)$ and observe that the following identity is immediate: 
\[
R_k(\theta) = \underbrace{ \wh G_k\l(\wh e_N(\theta) ;\theta \r) }_{=0} + \pd_z G_k(z;\theta)\big|_{z=0} \wh e_N(\theta) - \l(  \wh G_k\l(\wh e_N(\theta) ; \theta \r) - \wh G_k(0;\theta) \r).
\]
For any $\theta\in \Theta$ and $j=1,\ldots,k$, there exists $\tau_j = \tau_j(\theta)\in [0,1]$ such that 
\ml{
\rho' \l( \frac{\sqrt{n}}{\Delta_n}\l(\bar L_j(\theta,\theta_0) - L(\theta,\theta_0) - \wh e_N(\theta) \r) \r) = 
\rho' \l( \frac{\sqrt{n}}{\Delta_n}\l(\bar L_j(\theta,\theta_0) - L(\theta,\theta_0)  \r) \r)
\\
- \frac{\sqrt{n}}{\Delta_n} \rho''\l(  \frac{\sqrt{n}}{\Delta_n}\l(\bar L_j(\theta,\theta_0) - L(\theta,\theta_0)  \r)\r)\cdot \wh e_N(\theta) 
\\
+ \frac{n}{\Delta_n^2} \rho''' \l( \frac{\sqrt{n}}{\Delta_n}\l(\bar L_j(\theta,\theta_0) - L(\theta,\theta_0) -\tau_j \wh e_N(\theta) \r) \r)\cdot \l( \wh e_N(\theta) \r)^2.
}
Therefore, 
\ml{
\wh G_k\l(\wh e_N(\theta) ; \theta \r) - \wh G_k(0;\theta) 
= - \frac{\sqrt{n}}{k\Delta_n} \sum_{j=1}^k \rho''\l(  \frac{\sqrt{n}}{\Delta_n}\l(\bar L_j(\theta,\theta_0) - L(\theta,\theta_0)  \r)\r)\cdot \wh e_N(\theta) 
\\
+ \frac{n}{k\Delta_n^2} \sum_{j=1}^k \rho''' \l( \frac{\sqrt{n}}{\Delta_n}\l(\bar L_j(\theta,\theta_0) - L(\theta,\theta_0)\r) \r)\cdot \l( \wh e_N(\theta) \r)^2
\\
+\frac{n}{k\Delta_n^2} \sum_{j=1}^k \l(\rho''' \l( \frac{\sqrt{n}}{\Delta_n}\l(\bar L_j(\theta,\theta_0) - L(\theta,\theta_0) -\tau_j \wh e_N(\theta) \r) \r) - \rho''' \l( \frac{\sqrt{n}}{\Delta_n}\l(\bar L_j(\theta,\theta_0) - L(\theta,\theta_0) \r) \r)\r)\cdot \l( \wh e_N(\theta) \r)^2
}
and
\mln{
\label{eq:R'}
R_k(\theta) =  \frac{\sqrt{n}}{\Delta_n} \frac{1}{k}\sum_{j=1}^k \l( \rho''\l( \frac{\sqrt{n}}{\Delta_n}\l(\bar L_j(\theta,\theta_0) - L(\theta,\theta_0)  \r)\r) - \mb E  \rho''\l(  \frac{\sqrt{n}}{\Delta_n}\l(\bar L_j(\theta,\theta_0) - L(\theta,\theta_0)  \r)\r) \r) \cdot \wh e_N(\theta)
\\
- \frac{n}{\Delta_n^2} \frac{1}{k}\sum_{j=1}^k \rho''' \l( \frac{\sqrt{n}}{\Delta_n}\l(\bar L_j(\theta,\theta_0) - L(\theta,\theta_0) \r) \r)\cdot \l( \wh e_N(\theta) \r)^2
\\
- \frac{n}{\Delta_n^2} \frac{1}{k}\sum_{j=1}^k \l(\rho''' \l( \frac{\sqrt{n}}{\Delta_n}\l(\bar L_j(\theta,\theta_0) - L(\theta,\theta_0) -\tau_j \wh e_N(\theta) \r) \r) \r.
\\
\l. - \rho''' \l( \frac{\sqrt{n}}{\Delta_n}\l(\bar L_j(\theta,\theta_0) - L(\theta,\theta_0) \r) \r)\r)\cdot \l( \wh e_N(\theta) \r)^2
= R'(\theta) + R''(\theta) + R'''(\theta).
}
It follows from Lemma \ref{lemma:unif} (with $\m O = 0$) and Lemma \ref{lemma:sup-power} that
\[
\sup_{ \| \theta-\theta_0 \| \leq \delta} \l| \wh e_N(\theta) \r| \leq C(d,\theta_0) \l( \frac{\delta}{\sqrt N} s +\frac{ \delta^2}{\sqrt n}\r)
\]
with probability at least $1-s^{-1}$ whenever $s\lesssim \sqrt k$. 
Moreover, Proposition \ref{lemma:denom} combined with Lemma \ref{lemma:sup-power} yields that 
\ml{
\sup_{\|\theta - \theta_0\|\leq \delta} \l|\frac{1}{k}\sum_{j=1}^k \l( \rho''\l(  \frac{\sqrt{n}}{\Delta_n}\l(\bar L_j(\theta,\theta_0) - L(\theta,\theta_0)  \r)\r) - \mb E  \rho''\l(  \frac{\sqrt{n}}{\Delta_n}\l(\bar L_j(\theta,\theta_0) - L(\theta,\theta_0)  \r)\r) \r)\r| 
\\
\leq C(d,\theta_0) \, \frac{\delta}{\sqrt k}s 
}
with probability at least $1-s^{-1}$. Therefore, the first term $R'(\theta)$ in \eqref{eq:R'} satisfies 
\[
\sup_{\|\theta - \theta_0\|\leq \delta} \l| R'(\theta) \r| \leq C(d,\theta_0)\l(  \frac{\delta^2}{k}s^2  + 
\frac{\delta^3}{\sqrt k}s\r)
\] 
on event $\m E$ of probability at least $1-\frac{2}{s}$. Observe that
\ml{
\sup_{\|\theta - \theta_0\|\leq \delta}\l|\frac{1}{k}\sum_{j=1}^k \l(\rho''' \l( \frac{\sqrt{n}}{\Delta_n}\l(\bar L_j(\theta,\theta_0) - L(\theta,\theta_0) \r) \r) - \mb E \rho''' \l( \frac{\sqrt{n}}{\Delta_n}\l(\bar L_j(\theta,\theta_0) - L(\theta,\theta_0) \r) \r)\r) \r| 
\\
\leq  C(d,\theta_0) \, \frac{\delta}{\sqrt k}s
}
with probability at least $1-s^{-1}$, again by Proposition \ref{lemma:denom}, and 
$\l|  \mb E \rho''' \l( \frac{\sqrt{n}}{\Delta_n}\l(\bar L_j(\theta,\theta_0) - L(\theta,\theta_0) \r) \r) \r| \leq C\delta^2$ by Lemma \ref{lemma:lindeberg}. Therefore, the term $R''(\theta)$ admits an upper bound 
\[
\sup_{\|\theta - \theta_0\|\leq \delta} \l| R''(\theta) \r| 
\leq C\l(  \frac{\delta^3}{k^{3/2}} s^3 + \sigma^6(\delta)\r)
\]
which holds with probability at least $1-s^{-1}$. 
Finally, as $\rho'''$  is Lipschitz continuous by assumption, the third term $R'''(\theta)$ can be estimated via
\[
\sup_{\|\theta - \theta_0\|\leq \delta} \l| R'''(\theta) \r| \leq C\frac{n}{\Delta_n^2} \l| \wh e_N(\theta)\r|^3 
\leq \frac{C}{\sqrt n}\l(  \frac{\delta^3}{k^{3/2}} s^3 + \sigma^6(\delta) \r)
\]
on event $\m E$ (note that this upper bound is smaller than the upper bound for $\sup_{\|\theta - \theta_0\|\leq \delta} \l| R''(\theta) \r|$ by the multiplicative factor of $\sqrt{n}$). 
Combining the estimates above and excluding all the higher order terms (taking into account the fact that $s\leq c(d,\theta_0)\sqrt{k}$), it is easy to conclude that 
\begin{equation*}
\sqrt{N}\sup_{\|\theta-\theta_0\|\leq \delta} \l| \frac{R_k(\theta)}{ \pd_z G_k(z;\theta)|_{z=0}} \r| 
\leq C(d,\theta_0)\Bigg( \delta^2 \frac{s^2}{\sqrt{k}}  + 
s\,\delta^3
 + \sqrt{k}\delta^6 \Bigg)
\end{equation*}
with probability at least $1-\frac{3}{s}$.
\end{proof}

\section*{Acknowledgements.}

The simulation results in this paper partially rely on the code created by Timoth\'{e}e Mathieu (\href{https://github.com/TimotheeMathieu/Excess-risk-bounds-in-robust-empirical-risk-minimization/}{https://github.com/TimotheeMathieu/Excess-risk-bounds-in-robust-empirical-risk-minimization/}), and the author wants to express his gratitude to Timoth\'{e}e for allowing him to use it.

\bibliography{MoM,robustERM} 
\bibliographystyle{apalike}

\appendix
\section{Auxiliary results.}
\label{sec:aux}

\subsection{Existence of solutions.}

In this section, we discuss simple sufficient conditions for existence of the estimator $\htheta_{n,k}$ defined in display \eqref{eq:M-est-erm2}. 
\begin{proposition}
Assume that $\Theta\subset \mb R^d$ is compact and that $\ell(\theta,x)$ is continuous with respect to the first variable for P-almost all $x$. Moreover, let $\rho$ be a convex function such that $\rho''(x)>0$ for all $x\in \mb R$. Then $\htheta_{n,k}$ exists.
\end{proposition}
\begin{proof}
It suffices to show that $\hL(\theta,\theta')$ is continuous. The existence claim then easily follows as $\hL(\theta,\theta')$ must be uniformly continuous on $\Theta\times \Theta$ due to compactness, which in turn implies, via a standard argument, continuity of the function $\theta\mapsto \max_{\theta'\in \Theta} \hL(\theta,\theta')$, hence the existence of $\htheta_{n,k}$, again in view of compactness. To establish the continuity of $\theta\mapsto \max_{\theta'\in \Theta} \hL(\theta,\theta')$ when $\hL(\theta,\theta')$ is uniformly continuous, note that for any $\theta\in \Theta$ and any $\eps>0$, $\hL(\theta,\theta') - \eps \leq \hL(\tilde\theta,\theta')\leq \hL(\theta,\theta')+\eps$ for all $\theta'\in\Theta$ as long as $\|\tilde\theta - \theta\|\leq \delta(\eps)$. It easily implies that 
$ \max_{\theta'\in \Theta} \hL(\theta,\theta')-\eps \leq \max_{\theta'\in \Theta} \hL(\tilde\theta,\theta') \leq \max_{\theta'\in \Theta} \hL(\theta,\theta')+\eps$, and the conclusion follows.

\noindent All that remains is to establish the continuity of $\hL(\theta,\theta')$. To this end, fix $\eps>0$ and let 
\[
R(z;\theta,\theta') =\frac{1}{k} \sum_{j=1}^k \rho\l(\sqrt{n}\,\frac{\bL_j (\theta) - \bL_j(\theta') - z}{\Delta_n}\r).
\] 
Since $R'(z;\theta,\theta')$ is strictly increasing in $z$, there exist $z_+(\eps)$ and $z_-(\eps)$ such that $R'(z_+(\eps);\theta,\theta') = \eps$ and $R'(z_-(\eps);\theta,\theta')=-\eps$. In particular, $\hL(\theta,\theta')\in (z_-(\eps),z_+(\eps))$. 
As $R''(\hL(\theta,\theta');\theta,\theta')>0$ in view of the assumption $\rho''>0$, $|z_+(\eps) - z_-(\eps)|\to 0$ as $\eps\to 0$. Since $\bL_j(\theta)-\bL_j(\theta')$ is continuous in $\theta,\theta'$ by assumption, $R$ is continuous in $\theta,\theta'$ as well, hence $\l|R(z_+(\eps);\tilde\theta,\tilde\theta') - R(z_+(\eps);\theta,\theta')\r|<\eps$ and $\l|R(z_-(\eps);\tilde\theta,\tilde\theta') - R(z_-(\eps);\theta,\theta')\r|<\eps$ whenever $\|(\theta,\theta') - (\tilde\theta,\tilde\theta')\|\leq \delta(\eps)$ for some $\delta(\eps)$ small enough. In this case, we see that the inequalities $R(z_+(\eps);\tilde\theta,\tilde\theta')>0$ and $R(z_-(\eps);\tilde\theta,\tilde\theta')<0$ hold, hence $\hL(\tilde\theta,\tilde\theta')\in (z_-(\eps),z_+(\eps))$, implying that $\l|\hL(\tilde\theta,\tilde\theta') -\hL(\theta,\theta')\r|\leq |z_+(\eps) - z_-(\eps)|\to 0$ as $\eps\to 0$, yielding the desired conclusion.
\end{proof}
We remark that elsewhere in this work, we choose $\rho$ with the second derivative vanishing outside of a neighborhood of $0$.  However, $R''(\hL(\theta,\theta');\theta,\theta')>0$ holds with high probability uniformly over $\theta,\theta'\in \Theta$ when 
$\Theta$ is compact and the class $\{\ell(\theta,\cdot), \ \theta\in\Theta\}$ satisfies assumptions of the paper. We sketch the steps needed to show this fact; all the required tools have already been established in the paper. First, note that in view of Lemma \ref{lemma:unif} and the triangle inequality, $\sup_{\theta,\theta'\in \Theta}\l| 
\hL(\theta,\theta') - L(\theta,\theta') \r| = O_P(n^{-1/2})$ as $n,k\to\infty$ with high probability, hence $\inf_{\theta,\theta'} R''(\hL(\theta,\theta');\theta,\theta')\geq \inf_{\theta,\theta',|z|\leq D/\sqrt{n}} R''(L(\theta,\theta')+z;\theta,\theta')$ for a large constant $D$, again with high probability. Next, the relation
\[
\frac{1}{n}\sup_{\theta,\theta',|z|\leq D/\sqrt{n}} \l| R''(L(\theta,\theta')+z;\theta,\theta') - \mb E R''(L(\theta,\theta')+z;\theta,\theta')\r| = o_P(1)
\] 
as $n,k\to\infty$ follows from an argument identical to the one used to prove Proposition \ref{lemma:denom} and Lemma \ref{lemma:sup-power}. Finally, $\mb E \rho''\l(\sqrt{n}\,\frac{\bL_j (\theta) - \bL_j(\theta') - L(\theta,\theta') - z}{\Delta_n}\r)= \mb E\rho''\l( \frac{Z(\theta,\theta') - z\sqrt{n}}{\Delta_n} \r) + o(1)$ in view of Lemma \ref{lemma:lindeberg}, where $Z(\theta,\theta')$ is a centered and normally distributed random variable with variance $\sigma^2(\theta,\theta')$. As $\rho''(x)\geq I\{|x|\leq 1\}$, we see that $\inf_{\theta,\theta',|z|\leq D/\sqrt{n}} \mb E\rho''\l( \frac{Z(\theta,\theta') - z\sqrt{n}}{\Delta_n} \r)>0$, yielding the result.


\subsection{Technical lemmas.}
\begin{lemma}
\label{lemma:lindeberg}
Let $F:\mb R\mapsto \mb R$ be a function such that $F''$ is bounded and Lipschitz continuous. Moreover, suppose that $\xi_1,\ldots,\xi_n$ are independent centered random variables such that $\mb E|\xi_j|^{2}<\infty$ for all $j$, and that $Z_j, \ j=1,\ldots,n$ are independent with normal distribution $N\l( 0, \var(\xi_j)\r)$. 
Then 
\[
\l|\mb E F\l( \sum_{j=1}^n \xi_j\r) - \mb E F\l( \sum_{j=1}^n Z_j\r)\r| \leq C(F) \sum_{j=1}^n \mb E \l[ \xi_j^2 \cdot\min(|\xi_j,1)\r].
\]
In particular, if $\mb E|\xi_j|^{2+\tau}<\infty$ for some $\tau\in (0,1]$ and all $j$, then 
\[
\l|\mb E F\l( \sum_{j=1}^n \xi_j\r) - \mb E F\l( \sum_{j=1}^n Z_j\r)\r| \leq C(F) \sum_{j=1}^n \mb E |\xi_j|^{2+\tau}.
\]
\end{lemma}
\begin{proof}
We will apply the standard Lindeberg's replacement method \citep[see e.g.][chapter 11]{o2014analysis}. 
For $1\leq j\leq n+1$, define $T_j:=F\l( \sum_{i=1}^{j-1} \xi_i + \sum_{i=j}^n Z_j\r)$. Then 
\[
\l| \mb E F\l( \sum_{j=1}^n \xi_j\r) - \mb E F\l( \sum_{j=1}^n Z_j\r) \r| =\l| \mb E T_{n+1} - \mb ET_1\r|
\leq \sum_{j=1}^{n} |\mb ET_{j+1} - \mb ET_{j}|.
\]
Moreover, Taylor's expansion formula gives that there exists (random) $\mu \in [0,1]$ such that
\ml{
T_{j+1} = F\l( \sum_{i=1}^{j-1} \xi_i + \sum_{i=j+1}^n Z_j\r) + F'\l( \sum_{i=1}^{j-1} \xi_i + \sum_{i=j+1}^n Z_j\r)\xi_j 
+ F''\l( \sum_{i=1}^{j-1} \xi_i + \sum_{i=j+1}^n Z_j\r) \frac{\xi_j^2}{2} 
\\
+ \l( F''\l( \sum_{i=1}^{j-1} \xi_i + \sum_{i=j+1}^n Z_j + \mu \xi_j\r) -  F''\l( \sum_{i=1}^{j-1} \xi_i + \sum_{i=j+1}^n Z_j\r) \r)\frac{\xi_j^2}{2}.
}
Similarly, 
\ml{
T_{j} = F\l( \sum_{i=1}^{j-1} \xi_i + \sum_{i=j+1}^n Z_j\r) + F'\l( \sum_{i=1}^{j-1} \xi_i + \sum_{i=j+1}^n Z_j\r)Z_j 
+ F''\l( \sum_{i=1}^{j-1} \xi_i + \sum_{i=j+1}^n Z_j\r) \frac{Z_j^2}{2} 
\\
+ \l( F''\l( \sum_{i=1}^{j-1} \xi_i + \sum_{i=j+1}^n Z_j + \mu' Z_j\r) -  F''\l( \sum_{i=1}^{j-1} \xi_i + \sum_{i=j+1}^n Z_j\r) \r)\frac{Z_j^2}{2}.
}
Lipschitz continuity and boundedness of $F''$ imply that $\l| F''(x) - F''(y) \r|\leq C(F) \min(1, |x-y|)$ with 
$C(F)=\max\l(2\|F\|_\infty,L(F'')\r)$. Therefore, 
\ml{
\l| \mb ET_{j+1} - \mb E T_j \r| \leq \l| \mb E \l( F''\l( \sum_{i=1}^{j-1} \xi_i + \sum_{i=j+1}^n Z_j + \mu \xi_j\r) -  F''\l( \sum_{i=1}^{j-1} \xi_i + \sum_{i=j+1}^n Z_j\r) \r)\frac{\xi_j^2}{2}\r|
\\
+ \l| \mb E \l( F''\l( \sum_{i=1}^{j-1} \xi_i + \sum_{i=j+1}^n Z_j + \mu' Z_j\r) -  F''\l( \sum_{i=1}^{j-1} \xi_i + \sum_{i=j+1}^n Z_j\r) \r)\frac{Z_j^2}{2}\r| 
\\
\leq C_1(F)\mb E\l[ \xi_j^2 \min(|\xi_j|,1) \r],
}
and the first claim follows. 
To establish the second inequality, it suffices to observe that for all $j$,
$\mb E\l[ \xi_j^2 \min(|\xi_j|,1) \r] = \mb E |\xi_j|^3 I\{|\xi_j|\leq1 \} + \mb E |\xi_j|^2 I\{|\xi_j|> 1\}$. 
Clearly, $|\xi_j|^3\leq  |\xi_j|^{2+\tau}$ on the event $\{|\xi_j|\leq 1\}$, whereas $|\xi_j|^2\leq  |\xi_j|^{2+\tau}$ on the event $\{|\xi_j|>1\}$. 
\end{proof}

\begin{lemma}
\label{lemma:sup-power}
Let $\m F = \l\{f_\theta, \ \theta\in \Theta'\subseteq \mb R^d\r\}$ be a class of functions that is Lipschitz in parameter, meaning that 
$|f_{\theta_1}(x) - f_{\theta_2}(x)|\leq M(x)\|\theta_1 - \theta_2\|$. Moreover, assume that $\mb EM^{p}(X)<\infty$ for some $p\geq 1$. Then 
\[
\mb E \sup_{\theta_1, \theta_2\in \Theta'} \l(\frac{1}{\sqrt{n}} \l| \sum_{j=1}^{n}  \l(f_{\theta_1}(X_j) - f_{\theta_2}(X_j) - P(f_{\theta_1} - f_{\theta_2})\r)\r| \r)^p 
\leq C(p) d^{p/2} \diam^p(\Theta',\|\cdot\|) \mb E \|M \|^p_{L_2(\Pi_n)}
\]
and 
\ml{
\mb E \sup_{\theta\in \Theta'} \l(\frac{1}{\sqrt{n}} \l| \sum_{j=1}^{n}  \l(f_{\theta}(X_j) - P f_{\theta_1} \r)\r| \r)^p 
\\
\leq C(p) \l( d^{p/2} \diam^p(\Theta',\|\cdot\|) \mb E \|M \|^p_{L_2(\Pi_n)} + \mb E^{1\wedge \frac{p}{2}} \l|f_{\theta_0}(X) - P f_{\theta_0}\r|^{2\vee p} \r)
}
for any $\theta_0\in \Theta'$.
\end{lemma}
\begin{proof}
Symmetrization inequality yields that 
\begin{multline*}
\mb E \sup_{\theta_1, \theta_2\in \Theta'} \l(\frac{1}{\sqrt{n}} \l| \sum_{j=1}^{n}  \l(f_{\theta_1}(X_j) - f_{\theta_2}(X_j) - P(f_{\theta_1} - f_{\theta_2})\r)\r| \r)^p 
\\
\leq C(p)\mb E \sup_{\theta_1, \theta_2\in \Theta'} \l(\frac{1}{\sqrt{n}} \l| \sum_{j=1}^{n}  \eps_j\l(f_{\theta_1}(X_j) - f_{\theta_2}(X_j) \r)\r| \r)^p 
\\
= C(p)\mb E_X \mb E_\eps \sup_{\theta_1, \theta_2\in \Theta'} \l(\frac{1}{\sqrt{n}} \l| \sum_{j=1}^{n}  \eps_j\l(f_{\theta_1}(X_j) - f_{\theta_2}(X_j) \r)\r| \r)^p.
\end{multline*}
As the process $f\mapsto \frac{1}{\sqrt{n}}  \sum_{j=1}^{n}  \eps_j\l(f_{\theta_1}(X_j) - f_{\theta_2}(X_j) \r)$ is sub-Gaussian conditionally on $X_1,\ldots,X_n$, its (conditional) $L_p$-norms are equivalent to $L_1$ norm. Hence, Dudley's entropy bound (see Theorem 2.2.4 in \cite{wellner1}) implies that 
\begin{multline*}
\mb E_\eps \sup_{\theta_1, \theta_2\in \Theta'} \l(\frac{1}{\sqrt{n}} \l| \sum_{j=1}^{n}  \eps_j\l(f_{\theta_1}(X_j) - f_{\theta_2}(X_j) \r)\r| \r)^p 
\\
\leq C(p) \l( \mb E_\eps \sup_{\theta_1, \theta_2\in \Theta'} \frac{1}{\sqrt{n}} \l| \sum_{j=1}^{n}  \eps_j\l(f_{\theta_1}(X_j) - f_{\theta_2}(X_j) \r)\r|  \r)^p
\\
\leq C(p)\l( \int_{0}^{D_n(\Theta')} \log^{1/2}N(z,T_n,d_n) dz \r)^p,
\end{multline*}
where $d^2_n(f_{\theta_1},f_{\theta_2}) = \frac{1}{n}\sum_{j=1}^n \l( f_{\theta_1}(X_j) - f_{\theta_2}(X_j)\r)^2$, 
$T_n = \l\{ (f_{\theta}(X_1),\ldots, f_{\theta}(X_n)), \ \theta \in \Theta'\r\}\subseteq \mb R^n$ and $D_n(\Theta')$ is the diameter of $\Theta$ with respect to the distance $d_n$. 
As $f_{\theta}(\cdot)$ is Lipschitz in $\theta$, we have that 
$d^2_n(f_{\theta_1},f_{\theta_2}) \leq \frac{1}{n}\sum_{j=1}^n M^2(X_j) \|\theta_1 - \theta_2\|^2$, implying that $D_n(\Theta')\leq \|M\|_{L_2(\Pi_n)}\diam(\Theta',\|\cdot\|)$ and 
\[
\log N(z,T_n,d_n)\leq \log N\l(z/\| M \|_{L_2(\Pi_n)},\Theta',\|\cdot\|\r) \leq \log\l(C\frac{\diam(\Theta', \|\cdot\|)\, \|M \|_{L_2(\Pi_n)}}{z}\r)^d.
\] 
Therefore, 
\[
\l( \int_{0}^{D_n(\Theta')} \log^{1/2}N(z,T_n,d_n) dz \r)^p \leq C d^{p/2} \l(\diam(\Theta',\|\cdot\|) \cdot \| M \|_{L_2(\Pi_n)}\r)^p
\]
and 
\[
\mb E_X \mb E_\eps \sup_{\theta_1, \theta_2\in \Theta'} \l(\frac{1}{\sqrt{n}} \l| \sum_{j=1}^{n}  \eps_j\l(f_{\theta_1}(X_j) - f_{\theta_2}(X_j) \r)\r| \r)^p \leq Cd^{p/2} \diam^p(\Theta',\|\cdot\|) \mb E \|M \|_{L_2(\Pi_n)}^p.
\]
\end{proof}
\noindent Proof of the second bound follows from the triangle inequality 
\ml{
\mb E \sup_{\theta\in \Theta'} \l( \l| \sum_{j=1}^{n} \frac{1}{\sqrt{n}} \l(f_{\theta}(X_j) - P f_{\theta_1} \r)\r| \r)^p  
\leq 
C(p)\Bigg(  \mb E\l| \frac{1}{\sqrt{n}} \sum_{j=1}^{n}  (f_{\theta_0}(X_j) - Pf_{\theta_0})\r|^p
\\
+\mb E \sup_{\theta \in \Theta'} \l(\l| \frac{1}{\sqrt{n}} \sum_{j=1}^{n}  \l(f_{\theta}(X_j) - f_{\theta_0}(X_j) - P(f_{\theta} - f_{\theta_0})\r)\r| \r)^p \Bigg),
}
and Rosenthal's inequality \citep{ibragimov2001best} applied to the term \linebreak[4] $\mb E\l| \frac{1}{\sqrt{n}} \sum_{j=1}^{n}  (f_{\theta_0}(X_j) - Pf_{\theta_0})\r|^p$.

\begin{proposition}
\label{lemma:denom}
Let $\theta\in \Theta$, and set $\delta_0:=r(\theta)$, where $r(\theta)$ is defined in assumption \ref{ass:3}. Then for all $0<\delta\leq \delta_0$,
\ml{
\mb E\sup_{\|\theta' - \theta\|\leq \delta} \l| \frac{1}{k}\sum_{j=1}^k \rho'' \l( \frac{\sqrt{n}}{\Delta_n}\l(\bar L_j(\theta',\theta_0) - L(\theta',\theta_0)\r) \r) - \mb E \rho'' \l( \frac{\sqrt{n}}{\Delta_n}\l(\bar L_1(\theta',\theta_0) - L(\theta',\theta_0)\r) \r) \r| 
\\
\leq \frac{8}{\Delta_n \sqrt{k}}\mb E\sup_{\|\theta' - \theta\|\leq \delta} \l| \frac{1}{\sqrt{N}}\sum_{j=1}^N \l( \ell(\theta',X_j) - \ell(\theta_0,X_j) - L(\theta',\theta_0)\r) \r|
}
As a consequence,
\ml{
\sup_{\|\theta' - \theta\|\leq \delta} \l| \frac{1}{k}\sum_{j=1}^k \rho'' \l( \frac{\sqrt{n}}{\Delta_n}\l(\bar L_j(\theta',\theta_0) - L(\theta',\theta_0)\r) \r) - \mb E \rho'' \l( \frac{\sqrt{n}}{\Delta_n}\l(\bar L_1(\theta',\theta_0) - L(\theta',\theta_0)\r) \r) \r| 
\\
\leq \frac{8s}{\Delta_n \sqrt{k}} \mb E\sup_{\|\theta' - \theta\|\leq \delta} \l| \frac{1}{\sqrt{N}}\sum_{j=1}^N \l( \ell(\theta',X_j) - \ell(\theta_0,X_j) - L(\theta',\theta_0) \r) \r|
}
with probability at least $1-\frac{1}{s}$, where $C>0$ is an absolute constant. Moreover, the bound still holds if $\rho''$ is replaced by $\rho'''$, up to the change in constants. 
\end{proposition}
\begin{proof}
Let $\eps_1,\ldots,\eps_k$ be i.i.d. Rademacher random variables independent of $X_1,\ldots,X_N$, and note that by symmetrization and contraction inequalities for the Rademacher sums \citep{LT-book-1991},
\begin{multline*}
\mb E\sup_{\|\theta' - \theta\|\leq \delta} \l| \frac{1}{k}\sum_{j=1}^k \rho''\l( \frac{\sqrt{n}}{\Delta_n}\l(\bar L_j(\theta',\theta_0) - L(\theta',\theta_0)\r) \r) - \mb E \rho'' \l( \frac{\sqrt{n}}{\Delta_n}\l(\bar L_1(\theta',\theta_0) - L(\theta',\theta_0)\r) \r) \r|   \\
\leq 2 \, \mb E\sup_{\|\theta' - \theta\|\leq \delta} \frac{1}{k} \l| \sum_{j=1}^k \eps_j \l( \rho''\l( \frac{\sqrt{n}}{\Delta_n}\l(\bar L_j(\theta',\theta_0) - L(\theta',\theta_0)\r) \r) - \rho''(0) \r) \r|
\\
\leq \frac{4 L(\rho'')}{\Delta_n \sqrt{k}} \, \mb E\sup_{\|\theta' - \theta\|\leq \delta}\l| \sum_{j=1}^k \eps_j \frac{\sqrt{n}}{\sqrt{k}} \l(\bar L_j(\theta',\theta_0) - L(\theta',\theta_0)\r)\r|,
\end{multline*}
where we used the fact that $\phi(x):=\rho'' \l( \frac{\sqrt{n}}{\Delta_n} x \r) - \rho''\l( 0 \r)$ is Lipschitz continuous (in fact, assumption \ref{ass:1} implies that the Lipschitz constant is equal to $1$) and satisfies $\phi(0)=0$. 
Now, desymmetrization inequality \citep[Lemma 2.3.6 in][]{wellner1} implies that 
\ml{
\mb E \sup_{\|\theta' - \theta\|\leq \delta}\l| \sum_{j=1}^k \eps_j \frac{\sqrt{n}}{\sqrt{k}} \l(\bar L_j(\theta',\theta_0) - L(\theta',\theta_0)\r)\r|
\\ 
\leq 
\frac{2}{\sqrt{N}} \mb E\sup_{\|\theta' - \theta\|\leq \delta} \l| \sum_{j=1}^N \l( \ell(\theta',X_j) - \ell(\theta_0,X_j) - L(\theta',\theta_0)\r) \r|,
}
hence the claim follows. 

The fact that $\rho''$ can be replaced by $\rho'''$ follows along the same lines as $\rho'''$ is Lipschitz continuous and $\|\rho'''\|_\infty<\infty$ by assumption \ref{ass:1}.
\end{proof}

\section{Numerical experiment: logistic regression.}
\label{sec:logistic}

As a simple proof of concept, we implemented the gradient descent-ascent algorithm mentioned in section \ref{sec:computational} for the problem of logistic regression; for a detailed discussion of closely related methods, we refer the reader to \citep{lecue2017robust,mathieu2021excess}. 
In the present setup, the dataset consists of pairs $(Z_j,Y_j)\in \mb R^2\times\{\pm1\}$, where the marginal distribution of the labels is uniform on $\{\pm1\}$, while the conditional distributions of $Z_j$'s are normal, that is, $\mathrm{Law}\l(Z_1\,|\,Y_1=1\r) = \mathcal{N}\l((-1,-1)^T,4I_2\r)$, $\mathrm{Law}\l(Z\,|\,Y=-1\r) \sim \mathcal{N}\l((1,1),4I_2\r)$, and $\pr{Y=1}=\pr{Y=-1}=1/2$; here, $I_2$ stands for the $2\times2$ identity matrix. The loss function is defined as $\ell(\theta,Z,Y) = \log\l( 1 + e^{-Y\langle\theta,Z\rangle}\r), \ \theta\in \mb R^2$. 
The dataset includes $40$ outliers for which $Y_j\equiv 1$ and $Z\sim \mathcal{N}\l((25,10),0.25I_2\r)$. The sample of $500$ ``informative'' observations was generated, along with $40$ outliers, and we compared the performance of robust method proposed in this paper with the standard logistic regression, as implemented in the Scikit-learn package \citep{pedregosa2011scikit}, that is known to be sensitive to outliers. Results of the experiment are presented in figure \ref{fig:scatter_plots} and illustrate the robustness of proposed approach. 
\begin{figure}[h]
\centering
    \includegraphics[width=\textwidth]{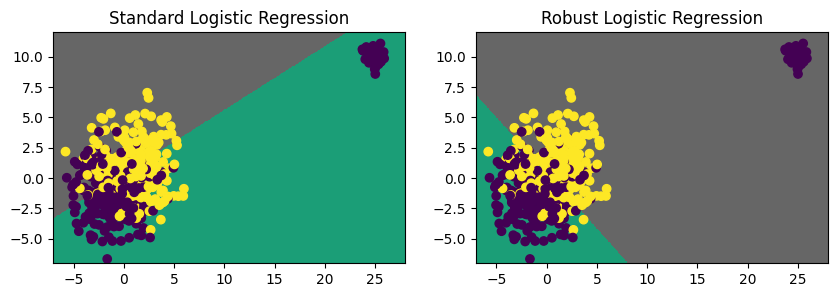}
\caption{Scatter plot of $N =540$ samples from the training dataset ($500$ informative observations and $40$ outliers). The color of the points correspond to their labels and the background color -- to the predicted labels (gray region corresponds to yellow labels and green -- to purple labels).
\label{fig:scatter_plots}}
\end{figure}

\end{document}